\RequirePackage{fix-cm}

\documentclass[smallextended]{svjour3}       
\smartqed  
\usepackage{graphicx}
\usepackage{amsmath}
\RequirePackage[colorlinks=true, pdfstartview=FitV, linkcolor=blue, citecolor=blue, urlcolor=blue]{hyperref}
\RequirePackage{hypernat}
\usepackage{amsfonts}
\usepackage{amssymb}
\usepackage{color}

\usepackage{mhequ}
\usepackage{fullpage}

\newcommand{\tr}{\mathrm{Trace}}

\newcommand{\F}{\mathcal{F}}
\newcommand{\RR}{\mathbb{R}}

\newcommand{\ZZ}{\mathbb{Z}}
\newcommand{\WW}{\mathbb{W}}
\newcommand{\eqdef}{\overset{{\mbox{\tiny  def}}}{=}}

\newcommand{\opnm}[3]{\|#1\|_{{\cal L}({\cal H}^{#2},{\cal H}^{#3})}}
\newcommand{\bra}[1]{\langle #1 \rangle}

\newcommand{\be}{\begin{equs}}
\newcommand{\ee}{\end{equs}}
\newcommand{\indic}{1{\hskip -2.5 pt}\hbox{I}}
\newcommand{\what}{\widehat}

\newcommand{\Normal}{\textrm{N}}

\newcommand{\longweak}{\Longrightarrow}
\newcommand{\dist}{\overset{\mathcal{D}}{\sim}}
\newcommand{\iid}{\text {i.i.d.}}
\newcommand{\EE}{\mathbb{E}}
\newcommand{\PP}{\mathbb{P}}

\newcommand\eps{\varepsilon}
\renewcommand\epsilon{\varepsilon}
\renewcommand\phi{\varphi}

\newcommand{\pphi}{\hat{\phi}}
\newcommand{\h}{\mathcal{H}}
\newcommand{\err}{\textbf{e}}

\newcommand{\lip}{\textrm{lip}}

\newtheorem{thm}[theorem]{Theorem}
\newtheorem{lem}[theorem]{Lemma}
\newtheorem{prop}[theorem]{Proposition}
\newtheorem{assumptions}[theorem]{Assumptions}

\newtheorem{conditions}[theorem]{Conditions}

\begin{document}

\title{Noisy Gradient Flow from a Random Walk in Hilbert Space
}


\author{
Natesh S. Pillai
\and
Andrew M. Stuart
\and
Alexandre H. Thi{\'e}ry
}


\institute{
Natesh S. Pillai \at
Department of Statistics,\\
Harvard University\\
1 Oxford Street, Cambridge\\ 02138, MA, USA\\
\email{pillai@stat.harvard.edu}           
\and
Andrew M. Stuart \at
Mathematics Institute\\
Warwick University\\ CV4 7AL, UK\\
\email{a.m.stuart@warwick.ac.uk}
\and
Alexandre H. Thi{\'e}ry\\
Department of Statistics\\
Warwick University\\ CV4 7AL, UK\\
\email{a.h.thiery@warwick.ac.uk}
}

\date{Received: date / Accepted: date}

\maketitle

\begin{abstract}
Consider a probability measure on a Hilbert space defined via
its  density with respect to a Gaussian.
The purpose of this paper is to demonstrate that an appropriately
defined Markov chain, which is reversible with respect to the
measure in question, exhibits a diffusion limit to a noisy
gradient flow, also reversible with respect to the same measure.
The Markov chain is defined by applying
a Metropolis-Hastings accept-reject mechanism \cite{Tie} to an
Ornstein-Uhlenbeck (OU) proposal which is itself reversible with respect to
the underlying Gaussian measure. The resulting noisy gradient flow
is a stochastic partial differential equation driven by a Wiener
process with spatial correlation given by the underlying Gaussian structure.

There are two primary motivations for this work. The first
concerns insight into Monte Carlo Markov Chain (MCMC) methods for
sampling of measures on a Hilbert space defined via a density
with respect to a Gaussian measure. These measures must be approximated on
finite dimensional spaces of dimension $N$ in order to be sampled.
A conclusion of the work herein is that MCMC methods based on
prior-reversible OU proposals will explore the target measure
in ${\mathcal O}(1)$ steps with respect to dimension $N$. This
is to be contrasted with standard MCMC methods based on the
random walk or Langevin proposals which require ${\mathcal O}(N)$ 
and ${\mathcal O}(N^{1/3})$ steps respectively 
\cite{mattingly2011spde,PST12}. The second motivation relates
to optimization.  There are many applications where it is of interest to
find global or local minima of a functional defined on 
an infinite dimensional Hilbert space. 
Gradient flow or steepest descent
is a natural approach to this problem, but in its 
basic form requires computation of a gradient which, 
in some applications, may be an expensive
or complex task. This paper
shows that a stochastic gradient descent described 
by a stochastic partial differential equation can emerge from certain
carefully specified Markov chains. This idea is well-known in the finite state 
\cite{kirkpatrick1983optimization,cerny1985thermodynamical}
or finite dimensional context \cite{geman1985bayesian,geman1986diffusions,chiang1987diffusion,holley1989asymptotics}.  
The novelty of the work in this paper is that the emergence of
the noisy gradient flow is developed on an infinite dimensional Hilbert space.
In the context of global optimization, when the noise level is also adjusted
as part of the algorithm, methods of the type studied here go
by the name of simulated--annealing; see the review 
\cite{bertsimas1993simulated} for further references. 
Although we do not consider adjusting the noise-level as part of the
algorithm, the noise strength is a tuneable parameter in our construction
and the methods developed here could potentially be used to study
simulated annealing in a Hilbert space setting.

The transferable idea behind this work is that conceiving
of algorithms directly in the infinite dimensional setting 
leads to methods which are robust to finite dimensional
approximation.  We emphasize that discretizing, and then
applying standard finite dimensional techniques
in $\RR^N$, to either sample or optimize, can lead to algorithms
which degenerate as the dimension $N$ increases.

\end{abstract}

\section{Introduction}
\label{sec:1}

There are many applications where it is of interest to
find global or local minima of a functional
\begin{equation}
\label{eq:I}
J(x)=\frac 12 \|C^{-1/2}x\|^2+\Psi(x)
\end{equation}
where $C$ is a self-adjoint, positive
and trace-class linear operator
on a Hilbert space $\Bigl(\h,\langle \cdot,\cdot\rangle,
\|\cdot\|\Bigr).$ Gradient flow or steepest descent
is a natural approach to this problem, but in its 
basic form requires computation of the gradient
of $\Psi$ which, in some applications, may be an expensive
or complex task. The purpose of this paper
is to show how a stochastic gradient descent described by a stochastic partial differential equation can emerge from certain
carefully specified random walks, when combined with
a Metropolis-Hastings accept-reject mechanism \cite{Tie}.
In the finite state 
\cite{kirkpatrick1983optimization,cerny1985thermodynamical}
or finite dimensional context \cite{geman1985bayesian,geman1986diffusions,chiang1987diffusion,holley1989asymptotics}  
this is a well-known idea, which goes by the name 
of simulated-annealing; see the review 
\cite{bertsimas1993simulated} for further references. 
The novelty of the work in this paper is that the theory 
is developed on an infinite dimensional Hilbert space,
leading to an algorithm which is robust to finite dimensional
approximation: we adopt the ``optimize then discretize''
viewpoint (see \cite{hinze2008optimization}, Chapter 3). 
We emphasize that discretizing, and then applying standard
finite dimensional techniques in $\RR^N$ to optimize, can lead to algorithms
which degenerate as $N$ increases; the diffusion limit
proved in \cite{mattingly2011spde} provides a concrete example
of this phenomenon for the standard random walk algorithm.

The algorithms we construct have two basic building
blocks: (i) drawing samples from the centred Gaussian measure
$N(0,C)$ and (ii) evaluating $\Psi$. By judiciously combining these
ingredients we generate (approximately) a noisy
gradient flow for $J$ with tunable temperature parameter
controlling the size of the noise.
In finite dimensions the basic idea 
is built from Metropolis-Hastings methods which have an
invariant measure with Lebesgue density proportional 
to $\exp\bigl(-\tau^{-1}J(x)\bigr)$.
The essential challenge in transferring these finite-dimensional algorithms to an infinite-dimensional setting is that there is no Lebesgue measure. This issue can be circumvented
by working with measures defined via their density with 
respect to a Gaussian measure, and 
for us the natural Gaussian measure  on $\h$
is
\begin{equs}
\label{eq:defpi}
\pi^{\tau}_0 = \Normal(0, \tau \, C).
\end{equs}
The quadratic form $\|C^{-\frac12}x\|^2$ is the square of the Cameron-Martin
norm corresponding to the Gaussian measure $\pi^{\tau}_0.$
Given $\pi^{\tau}_0$ we may then
define the (in general non-Gaussian) measure
$\pi^{\tau}$
via its Radon-Nikodym derivative with respect to $\pi^{\tau}$:
\begin{equs} \label{eqn:targmeas}
\frac{d\pi^{\tau}}{d\pi^{\tau}_0}(x) \propto \exp 
\Bigl( -\frac{\Psi(x)}{\tau}\Bigr).
\end{equs}
We assume that $\exp\bigl(-\tau^{-1}\Psi(\cdot)\bigr)$ is in $L^1_{\pi_0^{\tau}}$.
Note that if $\h$ is finite dimensional then $\pi^{\tau}$
has Lebesgue density proportional 
to $\exp\bigl(-\tau^{-1}J(x)\bigr).$

Our basic strategy will be to construct a Markov chain
which is $\pi^{\tau}$-invariant and to show that
a piecewise linear interpolant of the Markov chain
converges weakly (in the sense of probability measures)
to the desired noisy gradient flow in an appropriate
parameter limit.
To motivate the Markov chain we first observe
that the linear SDE in $\h$ given by
\begin{equs} \label{eqn:limit:SPDEL}
    dz   &= -z  \, dt + \sqrt{2\tau} dW\\
    z_0 &= x, 
\end{equs}
where $W$ is a Brownian motion in $\h$ 
with covariance operator equal to $C$, is reversible 
and ergodic with respect to $\pi_0^{\tau}$ 
given by \eqref{eq:defpi} \cite{da1996ergodicity}. 
If $t>0$ then the exact solution of this equation
has the form, for $\delta=\frac12(1-e^{-2t})$,
\begin{equs}
z(t)&=e^{-t}x+\sqrt{\Bigl(\tau(1-e^{-2t})\Bigr)}\xi\\
&=\bigl(1-2\delta\bigr)^{\frac12}x+\sqrt{2\delta\tau}\xi,
\label{eq:prop}
\end{equs}
where $\xi$ is a Gaussian random variable drawn from $N(0,C).$
Given a current state $x$ of our Markov chain we will propose
to move to $z(t)$ given by this formula, for some choice
of $t>0$. We will then accept or reject this proposed move
with probability found from pointwise evaluation of $\Psi$, 
resulting in a Markov chain $\{x^{k,\delta}\}_{k \in \ZZ^+}.$
The resulting Markov chain corresponds to the preconditioned Crank-Nicolson, 
or pCN,  method, also refered to as the PIA method with
$(\alpha,\theta)=(0,\frac12)$ in the paper \cite{Besk:etal:08}
where it was introduced; this is one 
of a family of Metropolis-Hastings
methods defined on the Hilbert space ${\cal H}$ and the
review \cite{CRSW13} provides further details.

From the output of the pCN Metropolis-Hastings
method we construct
a continuous interpolant of the Markov chain defined by 
\begin{equs}  \label{eqn:MCMCe}
z^{\delta}(t)= \frac{1}{\delta} \, (t-t_k) \, x^{k+1,\delta} 
+ \frac{1}{\delta} \, (t_{k+1}-t) \, x^{k,\delta} 
\qquad \text{for} \qquad 
t_k \leq t < t_{k+1}
\end{equs}
with $t_k \eqdef k \delta.$ 
The main result of the paper is that as $\delta \to 0$
the Hilbert-space valued function of time
$z^{\delta}$ converges weakly to $z$ solving the 
Hilbert space valued SDE, or SPDE, following the dynamics
\begin{equation} \label{eqn:limit:SPDEG}
dz   = -\Big( z + C \nabla \Psi(z) \Big) \, dt + \sqrt{2\tau} dW\\
\end{equation}
on pathspace. This equation is reversible and ergodic
with respect to the measure 
$\pi^{\tau}$ \cite{da1996ergodicity,hair:Stua:Voss:07}.
It is also known that small ball probabilities are
asymptotically maximized (in the small radius limit), 
under $\pi^{\tau}$, on balls
centred at minimizers of $J$ \cite{DLSV13}. The result
thus shows that the algorithm will generate sequences
which concentrate near minimizers of $J$.

Because the SDE \eqref{eqn:limit:SPDEG} does not possess
the smoothing property, almost sure fine scale properties under
its invariant measure $\pi^{\tau}$ are not necessarily 
reflected at any finite time. For example, if $C$ is the
covariance operator of Brownian motion or Brownian
bridge then the quadratic variation of draws from the invariant
measure, an almost sure quantity, is not reproduced at any
finite time in \eqref{eqn:limit:SPDEG} unless $z(0)$ has
this quadratic variation; the almost sure property is
approached asymptotically as $t \to \infty$. 
This behaviour is reflected in the
underlying Metropolis-Hastings Markov chain pCN 
with weak limit \eqref{eqn:limit:SPDEG}, where the
almost sure property is only reached asymptotically 
as $n \to \infty.$ In a second result
of this paper we will show that almost sure quantities 
such as the quadratic variation under pCN 
satisfy a limiting linear ODE
with globally attractive steady state given by the value
of the quantity under $\pi^{\tau}$. This gives quantitative
information about the rate at which the pCN algorithm
approaches statistical equilibrium.

We have motivated the limit theorem in this paper
through the goal of creating noisy gradient flow in
infinite dimensions with tuneable noise level,  using only 
draws from a Gaussian random variable and evaluation
of the non-quadratic part of the objective function. 
A second motivation for the work comes from understanding
the computational complexity of MCMC methods, and for this
it suffices to consider $\tau$ fixed at $1$. The paper 
\cite{mattingly2011spde} shows that discretization of
the standard Random Walk Metropolis algorithm, S-RWM, will also
have diffusion limit given by \eqref{eqn:limit:SPDEG} as the
dimension of the discretized space tends to infinity, whilst
the time increment $\delta$ in \eqref{eqn:MCMCe}, decreases
at a rate inversely proportional to $N$. The condition on
$\delta$ is a form of CFL condition, in the language of
computational PDEs, and implies that ${\cal O}(N)$ steps
will be required to sample the desired probability distribution.
In contrast the pCN method analyzed here has no CFL restriction:
$\delta$ may tend to zero independently of dimension; indeed
in this paper we work directly in the setting of infinite
dimension. The reader interested in this computational
statistics perspective on diffusion limits may also
wish to consult the paper \cite{PST12} which demonstrates that
the Metropolis adjusted Langevin algorithm, MALA, requires
a CFL condition which implies that ${\cal O}(N^{\frac13})$ steps
are required to sample the desired probability distribution. 
Furthermore, the formulation of the limit theorem that we prove
in this paper is closely related to the methodologies
introduced in \cite{mattingly2011spde} and \cite{PST12}; it should be mentioned   nevertheless that the analysis carried out in this article allows to prove a diffusion limit for a sequence of Markov chains evolving in a possibly non-stationary regime. This was not the case in \cite{mattingly2011spde} and \cite{PST12}.

We prove in Theorem \ref{thm:main} that for a fixed temperature parameter $\tau>0$, as the time increment $\delta$ goes to $0$, the pCN algorithm behaves as a stochastic gradient descent.
By adapting the temperature $\tau \in (0,\infty)$ 
according to an appropriate cooling schedule it is possible
to locate global minima of $J$; standard heuristics show that the distribution $\pi^{\tau}$ concentrates on a $\tau^{1/2}$-neighbourhood around the global minima of the functional $J$. We stress though that all the proofs presented in this article assume a constant temperature. The asymptotic analysis of the effect of the cooling schedule is left for future work; the study of such Hilbert space valued simulated annealing algorithms presents several challenges, one of them being that that the probability distributions $\pi^{\tau}$ are mutually singular for different temperatures $\tau>0$.

In section \ref{sec:2} we describe some notation used
throughout the paper, discuss the required properties
of Gaussian measures and Hilbert-space valued Brownian
motions, and state our assumptions. Section \ref{sec:3}
contains a precise definition of the Markov chain 
$\{x^{k,\delta}\}_{k \geq 0}$,
together with statement and proof
of the weak convergence theorem that is the main result
of the paper. Section \ref{sec:4} contains proof
of the lemmas which underly the weak convergence
theorem. 
In section \ref{sec:5} we state and prove the limit
theorem for almost sure quantities such as quadratic
variation; such results are often termed ``fluid limits''
in the applied probability literature.
An example is presented in section \ref{sec:6}.
We conclude in section \ref{sec:7}.

\section{Preliminaries}
\label{sec:2}

In this section we define some notational conventions, 
Gaussian measure and Brownian motion in Hilbert space, and
state our assumptions concerning the operator $C$ and
the functional $\Psi.$ 

\subsection{Notation}

Let $\Bigl(\h, \bra{\cdot, \cdot}, \|\cdot\|\Bigr)$ 
denote a separable Hilbert space of real valued functions with
the canonical norm derived from the inner-product. 
Let $C$ be a positive symmetric trace class operator on $\h$ 
and $\{\phi_j,\lambda^2_j\}_{j \geq 1}$ be the eigenfunctions
and eigenvalues of $C$ respectively, so that
$C\phi_j = \lambda^2_j \,\phi_j$ 
for $j \in \mathbb{N}$.
We assume a normalization under which $\{\phi_j\}_{j \geq 1}$ 
forms a complete orthonormal basis in $\h$.
For every $x \in \h$ we have the representation
$x = \sum_{j} \; x_j \phi_j$ where $x_j=\langle x,\phi_j\rangle.$
Using this notation, we define Sobolev-like spaces $\h^r, r \in \RR$, with the
inner-products and norms defined by
\begin{equs}\label{eqn:Sob}
\langle x,y \rangle_r\eqdef\sum_{j=1}^\infty j^{2r}x_jy_j
\qquad \text{and} \qquad
\|x\|^2_r \eqdef \sum_{j=1}^\infty j^{2r} \, x_j^{2}.
\end{equs}
Notice that $\h^0 = \h$. Furthermore
$\h^r \subset \h \subset \h^{-r}$ and 
$\{j^{-r} \phi_j\}_{j \geq 1}$ is an orthonormal 
basis of $\h^r$ for any $r >0$.  
For a positive, self-adjoint operator $D : \h^r \mapsto \h^r$, its trace in
$\h^r$ is defined as
\begin{equs}
\tr_{\h^r}(D) \;\eqdef\; \sum_{j=1}^\infty \bra{ (j^{-r} \phi_j), D
(j^{-r} \phi_j) }_r.
\end{equs}
Since $\tr_{\h^r}(D)$ does not depend on the orthonormal basis $\{\phi_j\}_{j \geq 1}$,
the operator $D$ is said to be trace class in $\h^r$ if $\tr_{\h^r}(D) < \infty$
for
some, and hence any, orthonormal basis of $\h^r$.
Let $\otimes_{\h^r}$
denote the outer product operator in $\h^r$
defined by
\begin{equs}\label{eqn:outprod}
(x \otimes_{\h^r} y) z \eqdef \bra{ y, z}_r \,x 
\end{equs}
for vectors $x,y,z \in \h^r$.
For an operator $L: \h^r \mapsto \h^l$, we denote its operator norm 
by $\opnm{\cdot}{r}{l}$ defined by
$\opnm{L}{r}{l} \eqdef \sup \, \big\{ \|L x\|_l 
,: \|x\|_r=1 \big\}$.
For self-adjoint $L$ and $r=l=0$ this is, of course, the spectral radius of $L$.
Throughout we use the following notation.
\begin{itemize}
\item Two sequences $\{\alpha_n\}_{n \geq 0}$ and $\{\beta_n\}_{n \geq 0}$
satisfy $\alpha_n \lesssim \beta_n$ 
if there exists a constant $K>0$ satisfying $\alpha_n \leq K \beta_n$ for all $n
\geq 0$.
The notations $\alpha_n \asymp \beta_n$ means that $\alpha_n \lesssim \beta_n$
and $\beta_n \lesssim \alpha_n$.

\item 
Two sequences of real functions $\{f_n\}_{n \geq 0}$ and $\{g_n\}_{n \geq 0}$
defined on the same set $\Omega$
satisfy $f_n \lesssim g_n$ if there exists a constant $K>0$ satisfying $f_n(x)
\leq K g_n(x)$ for all $n \geq 0$
and all $x \in \Omega$.
The notations $f_n \asymp g_n$ means that $f_n \lesssim g_n$ and $g_n \lesssim
f_n$.

\item
The notation $\EE_x \big[ f(x,\xi) \big]$ denotes expectation 
with variable $x$ fixed, while the randomness present in $\xi$
is averaged out.

\item We use the notation $a \wedge b$ instead of $\min(a,b)$.
\end{itemize}

\subsection{Gaussian Measure on Hilbert Space}
The following facts concerning Gaussian measures on
Hilbert space, and Brownian motion in Hilbert space,
may be found in \cite{Dapr:Zaby:92}.
Since $C$ is self-adjoint, positive and trace-class
we may associate with it a centred Gaussian measure
$\pi_0$ on $\h$ with covariance operator $C$, 
\textit{i.e.,} $\pi_0 \eqdef \Normal(0,C)$. If $x \dist \pi_0$ then we may write
its Karhunen-Lo\'eve expansion,
\begin{equs}
x = \sum_{j=1}^\infty \lambda_j\,\rho_j \,\phi_j, 
\label{eqn:KLexp}
\end{equs}
with $\{\rho_j\}_{j \geq 1}$ an i.i.d sequence 
of standard centered Gaussian random variables; 
since $C$ is trace-class,
the above sum converges in $L^2$. Notice that 
for any value of $r \in \RR$ we have 
$\EE \|X\|_r^2 = \sum_{j \geq 1} j^{2r} \bra{X,\phi_j}^2 = \sum_{j \geq 1} j^{2r} \lambda_j^2$ 
for $X \dist \pi_0$. For values of $r \in \RR$ such that $\EE \|X\|_r^2 < \infty$ we indeed have 
$\pi_0\big( \h^r \big)=1$ and the random variable $X$ can 
also be described as a Gaussian random variable in $\h^r$. One can readily check that in this case the covariance operator $C_r:\h^r \to \h^r$ of $X$ when viewed as a $\h^r$-valued random variable is given by
\begin{equs}\label{eqn:covopcr}
C_r = B^{1/2}_r \,C \,B^{1/2}_r.
\end{equs}
where 
$B_r : \h \mapsto \h$ denote the operator which is
diagonal in the basis $\{\phi_j\}_{j \geq 1}$ with diagonal entries
$j^{2r}$. In other words,
$B_r \,\phi_j = j^{2r} \phi_j$
so that $B^{\frac12}_r \,\phi_j = j^r \phi_j$ 
and $\EE\big[ \bra{X,u}_r \bra{X,v}_r \big] = \bra{u,C_r v}_r$ for $u,v \in \h^r$ and $X \dist \pi_0$. 
The condition $\EE \|X\|_r^2 < \infty$ can  equivalently be stated as
\begin{equs}
\tr_{\h^r}(C_r) < \infty.
\end{equs}
This shows that even though the Gaussian measure $\pi_0$ is 
defined on $\h$, depending on the decay of the eigenvalues of $C$, there exists
an entire range of values of $r$ such that 
$\EE\|X\|_r^2 = \tr_{\h^r}(C_r) < \infty$ and 
in that case the measure $\pi_0$ has full support on $\h^r$.

Frequently in applications the functional $\Psi$ arising
in \eqref{eq:I} may not
be defined on all of $\h$, but only on a subspace $\h^s \subset \h$,
for some exponent $s>0$. 
From now onwards we fix a distinguished exponent $s>0$ and assume that
$\Psi:\h^s \to \RR$
and that $\tr_{\h^s}(C_s) < \infty$ so that 
$\pi(\h^s)=\pi_0^{\tau}(\h^s)=\pi^{\tau}(\h^s)=1$; the change of measure formula \eqref{eqn:targmeas} is well defined.
For ease of notations we introduce
\begin{equs}
\hat{\phi}_j = B_s^{-\frac{1}{2}} \phi_j  = j^{-s} \, \phi_j
\end{equs}
so that the family $\{\pphi_j\}_{j \geq 1}$
forms an orthonormal basis for
$\big(\h^s, \bra{\cdot,\cdot}_s\big)$.
We may view the Gaussian measure $\pi_0 = \Normal(0,C)$ on $\big(\h,
\bra{\cdot,\cdot}\big)$ 
as a Gaussian measure $\Normal(0,C_s)$ on $\big(\h^s, \bra{\cdot,\cdot}_s\big)$.

A Brownian motion $\{W(t)\}_{t \geq 0}$ in $\h^s$ with covariance operator
$C_s:\h^s \to \h^s$ is a continuous Gaussian
process with stationary increments satisfying $\EE\big[ 
\bra{W(t),x}_s \bra{W(t),y}_s \big] = t \bra{x,C_s y}_s$.
For example, taking $\{\beta_j(t)\}_{j \geq 1}$ independent standard real
Brownian motions,
the process 
\begin{equs} \label{e.brow.expan}
W(t) = \sum_j (j^s \lambda_j) \, \beta_j(t) \pphi_j 
\end{equs}
defines a Brownian motion in $\h^s$ with covariance operator $C_s$;
equivalently, this same process $\{W(t)\}_{t \geq 0}$ 
can be described as a Brownian motion in $\h$ with covariance operator equal to
$C$ since Equation \eqref{e.brow.expan}
may also be expressed as $W(t) = \sum_{j=1}^{\infty} \lambda_j \beta_j(t)
\phi_j$.

\subsection{Assumptions} \label{sec:assumptions}
In this section we describe the assumptions on the covariance 
operator $C$ of the Gaussian measure $\pi_0 =\Normal(0,C)$ and the functional
$\Psi$, and the connections between them.
Roughly speaking we will assume that the second-derivative
of $\Psi$ is globally bounded as an operator acting between two spaces
which arise naturally from understanding the domain of the function $\Psi$;
furthermore the domain of $\Psi$ must be a set of full measure with
respect to the underlying Gaussian. If the eigenvalues of $C$ decay
like $j^{-2\kappa}$ and $\kappa>\frac12$ then $\pi_0^{\tau}(\h^s)=1$ for 
all $s<\kappa-\frac12$ and so we will assume eigenvalue decay of this
form and assume the domain of $\Psi$ is defined appropriately. We
now formalize these ideas.

For each $x \in \h^s$ the derivative $\nabla \Psi(x)$
is an element of the dual $(\h^s)^*$ of $\h^s$, comprising 
the linear functionals on $\h^s$.
However, we may identify $(\h^s)^* = \h^{-s}$ and view $\nabla \Psi(x)$
as an element of $\h^{-s}$ for each $x \in \h^s$. With this identification,
the following identity holds
\begin{equs}
\| \nabla \Psi(x)\|_{\mathcal{L}(\h^s,\RR)} = \| \nabla \Psi(x) \|_{-s} 
\end{equs}
and the second derivative $\partial^2 \Psi(x)$ can be identified with an
element 
of $\mathcal{L}(\h^s, \h^{-s})$.
To avoid technicalities we assume that $\Psi(x)$ is quadratically bounded, 
with first derivative linearly bounded and second derivative globally 
bounded. Weaker, localized assumptions 
could be dealt with by use of stopping time arguments. 
\begin{assumptions} \label{ass:1}
The functional $\Psi$ and the covariance operator $C$ satisfy the following
assumptions.
\begin{enumerate}
\item[A1.] {\bf Decay of Eigenvalues $\lambda_j^2$ of $C$:}
there exists a constant $\kappa > \frac{1}{2}$ such that
\begin{equs} \label{eqn:decayeigen}
\lambda_j \asymp j^{-\kappa}.
\end{equs}

\item[A2.] {\bf Domain of $\Psi$:}
there exists an exponent $s \in [0, \kappa - 1/2)$ such $\Psi$ is defined on
$\h^s$.

\item[A3.] {\bf Size of $\Psi$:} 
the functional $\Psi:\h^s \to \RR$ satisfies the growth conditions
\begin{equs}
0 \quad\leq\quad \Psi(x) \quad\lesssim\quad 1 +  \|x\|_s^2 . 
\end{equs}

\item[A4.] {\bf Derivatives of $\Psi$:} 
The derivatives of $\Psi$ satisfy
\begin{equs}
\| \nabla \Psi(x)\|_{-s} \quad\lesssim\quad 1 + \|x\|_s  
\qquad \text{and} \qquad
\opnm{\partial^2 \Psi(x)}{s}{-s} \quad\lesssim\quad 1.
\end{equs}
\end{enumerate}
\end{assumptions}

\begin{remark}
The condition $\kappa > \frac{1}{2}$ ensures that $\tr_{\h^r}(C_r)  < \infty$
for any $r < \kappa - \frac{1}{2}$: this implies that
$\pi^{\tau}_0(\h^r)=1$ 
for any $\tau > 0$ and  $r < \kappa - \frac{1}{2}$.
\label{rem:one}
\end{remark}

\begin{remark}
The functional $\Psi(x)  = \frac{1}{2}\|x\|_s^2$ is defined on $\h^s$ and satisfies Assumptions \ref{ass:1}. Its 
derivative at $x \in \h^s$
is given by $\nabla \Psi(x) = \sum_{j \geq 0} j^{2s} x_j \phi_j \in \h^{-s}$
with  
$\|\nabla \Psi(x)\|_{-s} = \|x\|_s$. The second derivative $\partial^2 \Psi(x)
\in \mathcal{L}(\h^s, \h^{-s})$ 
is the linear operator that maps $u \in \h^s$ to $\sum_{j \geq 0} j^{2s}
\bra{u,\phi_j} \phi_j \in \h^{-s}$:
its norm satisfies $\| \partial^2 \Psi(x) \|_{\mathcal{L}(\h^s, \h^{-s})} = 1$
for any $x \in \h^s$.
\end{remark}

\noindent
The Assumptions \ref{ass:1} ensure that the functional $\Psi$ behaves 
well in a sense made precise in the following lemma.

\begin{lem} \label{lem:lipshitz+taylor}
Let Assumptions \ref{ass:1} hold.
\begin{enumerate}
\item 
The function $d(x) \eqdef -\Big( x + C \nabla \Psi(x) \Big)$ is globally
Lipschitz on $\h^s$:
\begin{equs} \label{e.lipshitz}
\|d(x) - d(y)\|_s \;\lesssim \; \|x-y\|_s
\qquad \qquad \forall x,y \in \h^s.
\end{equs}

\item
The second order remainder term in the Taylor expansion of $\Psi$ satisfies
\begin{equs} \label{e.taylor.order2}
\big| \Psi(y)-\Psi(x) - \bra{\nabla \Psi(x), y-x} \big| \lesssim \|y-x\|_s^2
\qquad \qquad \forall x,y \in \h^s.
\end{equs}
\end{enumerate}
\end{lem}
\begin{proof}
See \cite{mattingly2011spde}.
\end{proof}

In order to provide a clean exposition, which highlights the
central theoretical ideas, we have chosen to make {\em global}
assumptions  on $\Psi$ and its derivatives. We believe that our
limit theorems could be extended to localized version of these
assumptions, at the cost of considerable technical complications
in the proofs, by means of stopping-time arguments.
The numerical example presented in section \ref{sec:6} corroborates this
assertion.  There are many applications which satisfy local versions 
of the assumptions given, including the Bayesian formulation of
inverse problems \cite{stuart2010inverse} and conditioned
diffusions \cite{Hair:Stua:Voss:10}.

\section{Diffusion Limit Theorem}
\label{sec:3}

This section contains a precise statement of the
algorithm, statement of the main theorem
showing that piecewise linear interpolant of the
output of the algorithm converges weakly to a noisy
gradient flow described by a SPDE, and proof of the main theorem. The proofs of
various technical lemmas are deferred to section \ref{sec:4}.

\subsection{pCN Algorithm}
\label{sec:algorithm}
We now define the Markov chain in $\h^s$ which
is reversible with respect to the measure 
$\pi^{\tau}$ given by Equation \eqref{eqn:targmeas}.
Let $x \in \h^s$ be the current position of the Markov chain.
The proposal candidate $y$ is given by \eqref{eq:prop}, so that
\begin{equs} \label{eqn:RWMprop}
y = \big(1 - 2\delta \big)^{\frac{1}{2}} x +  \sqrt{2\delta \tau} \,\xi 
\qquad\text{where} \qquad 
\xi \dist \Normal(0,C)
\end{equs}
and $\delta \in (0, \frac12)$ is a small parameter which we will send to
zero in order to obtain the noisy gradient flow.
In Equation \eqref{eqn:RWMprop}, 
the random variable $\xi$ is chosen \emph{independent} of $x$.
As described in \cite{Besk:etal:08}
(see also \cite{cottervariational,stuart2010inverse}), 
at temperature $\tau \in (0,\infty)$ 
the Metropolis-Hastings acceptance probability for
the proposal $y$ is given by
\begin{equs}\label{eqn:accprob}
\alpha^{\delta}(x,\xi) = 1 \wedge \exp\Bigl(-\frac{1}{\tau}\bigl(\Psi(y) -
\Psi(x)\bigr)  \Bigr).
\end{equs}
For future use, we define the local mean acceptance probability 
at the current position $x$ via the formula 
\begin{equs} \label{e.mean.loc.acc}
\alpha^{\delta}(x) \;=\; \EE_x\big[ \alpha^{\delta}(x,\xi) \big].
\end{equs}
The chain is then reversible with respect to $\pi^{\tau}$.
The Markov chain $x^{\delta} = \{x^{k,\delta}\}_{k \geq 0}$ 
can be written as
\begin{equs} \label{eqn:prop:bern}
\begin{aligned}
x^{k+1, \delta}
&=
\gamma^{k,\delta} y^{k,\delta} + (1- \gamma^{k,\delta})\,
x^{k, \delta}\\
y^{k,\delta}  
&=
\big( 1-2\delta \big)^{\frac12} \, x^{k,\delta} + \sqrt{2\delta
\tau} \xi^k 
\end{aligned}
\end{equs}
In the above equation, the $\xi^k$ are i.i.d Gaussian random variables $\Normal(0,C)$ and the
$\gamma^{k,\delta}$
are Bernoulli random variables which account for the accept-reject mechanism 
of the Metropolis-Hastings algorithm,
\begin{equs} \label{eqn:mark:dynamic}
\gamma^{k,\delta} 
\eqdef \gamma^{\delta}(x^{k,\delta},\xi^k, U^k) 
=
\indic_{\{ U_k < \alpha^{\delta}(x^{k,\delta},\xi^k)\}}
\;\dist\;
\textrm{Bernoulli}\Big(\alpha^{\delta}(x^{k,\delta},\xi^k)\Big).
\end{equs}
for an i.i.d sequence $\{U^k\}_{k \geq 0}$ of random variables uniformly distributed on the interval $(0,1)$ and independent from all the other sources of randomness.
The next lemma will be repeatedly used in the sequel. It states 
that the size of the jump $y-x$ is of order $\sqrt{\delta}$.

\begin{lem} \label{lem:size:prop}
Under Assumptions \ref{ass:1} and for any integer $p\geq 1$ the following
inequality 
\begin{equs}
\EE_x \big[ \|y-x\|^p_s \big]^{\frac{1}{p}} 
\quad \lesssim \quad \delta \, \|x\|_s + \sqrt{\delta}
\quad \lesssim \quad \sqrt{\delta}\; \big( 1 + \|x\|_s \big)
\end{equs}
holds for any $\delta \in (0, \frac12)$.
\end{lem}
\begin{proof}
The definition of the proposal \eqref{eqn:RWMprop} shows that
$ \|y-x\|^p_s \lesssim \delta^p \; \|x\|^p_s + \delta^{\frac{p}{2}} \, \EE\big[
\| \xi \|^p_s \big]$.
Fernique's theorem \cite{Dapr:Zaby:92} shows that $\xi$ has exponential moments
and therefore $\EE \big[ \|\xi\|^p_s \big] < \infty$. This gives the conclusion.
\end{proof}

\subsection{Diffusion Limit Theorem}
Fix a time horizon $T > 0$ and a temperature $\tau \in (0,\infty)$.
The piecewise linear interpolant $z^\delta$ of the Markov chain
\eqref{eqn:prop:bern} 
is defined by Equation \eqref{eqn:MCMCe}.
The following is the main result of this article.
Note that ``weakly'' refers to weak convergence of
probability measures.

\begin{thm}\label{thm:main}
Let Assumptions \ref{ass:1} hold. 
Let the Markov chain $x^{\delta}$ start at a fixed position $x_* \in \h^s$.
Then the sequence of processes $z^\delta$ converges weakly to 
$z$ in $C([0,T], \h^s)$, as $\delta \rightarrow 0$,
where $z$ solves the $\h^s$-valued stochastic differential equation
\begin{equs} \label{eqn:limit:SPDE0}
    dz   &= -\Big( z + C \nabla \Psi(z) \Big) \, dt + \sqrt{2\tau} \, dW\\
    z_0 &= x_* 
\end{equs}
and $W$ is a Brownian motion in $\h^s$ with covariance operator equal to $C_s$.
\end{thm}

\noindent
For conceptual clarity, we derive Theorem \ref{thm:main} as a consequence of the
general 
diffusion approximation Lemma \ref{lem:diff_approx}.
Consider a separable Hilbert space $\big( \h^s, \bra{\cdot, \cdot}_s \big)$ 
and a sequence of $\h^s$-valued Markov chains
$x^\delta = \{x^{k,\delta}\}_{k \geq 0}$.
The martingale-drift decomposition with time discretization $\delta$ of the
Markov chain $x^{\delta}$ reads
\begin{equs} \label{eqn:markov_chain_seq}
x^{k+1,\delta} 
&= x^{k,\delta} + \EE\big[ x^{k+1,\delta} - x^{k,\delta}\,|x^{k,\delta} \big] \\
&\qquad \qquad +
\Big( x^{k+1,\delta} - x^{k,\delta} - \EE\big[ x^{k+1,\delta} -
x^{k,\delta}\,|x^{k,\delta} \big]\Big) \\
&= x^{k,\delta} + d^{\delta}(x^{k,\delta}) \, \delta + \sqrt{2 \tau \delta} \;
\Gamma^{\delta}(x^{k,\delta}, \xi^k)
\end{equs}
where the approximate drift $d^{\delta}$ and volatility term
$\Gamma^{\delta}(x,\xi^k)$ are given by
\begin{equs} \label{eqn:approx:drift:noise}
d^{\delta}(x) &= \delta^{-1} \; \EE\big[ x^{k+1,\delta}-x^{k,\delta} \;|
x^{k,\delta}=x \big]\\
\Gamma^{\delta}(x,\xi^k) &= (2 \tau \delta)^{-1/2} \; \Big( x^{k+1,\delta} -
x^{k,\delta} 
- \EE\big[ x^{k+1,\delta} - x^{k,\delta} \;|x^{k,\delta}=x \big] \Big).
\end{equs}
In Equation \eqref{eqn:markov_chain_seq}, the conditional expectation $\EE\big[ x^{k+1,\delta} - x^{k,\delta}\,|x^{k,\delta} \big]$ is given by $\alpha^{\delta}(x^{k,\delta}, \xi^k) \times (y^{k, \delta} - x^{k,\delta})$ for a proposal $y^{k,\delta}$ and noise term $\xi^k$ as defined in Equation \eqref{eqn:markov_chain_seq}.
Notice that $\big\{ \Gamma^{k,\delta}\big\}_{k \geq 0}$, with $\Gamma^{k,\delta}
\eqdef \Gamma^{\delta}(x^{k, \delta}, \xi^k)$, is a martingale 
difference array in the sense that $M^{k,\delta} = \sum_{j=0}^k
\Gamma^{j,\delta}$ is a martingale adapted to the 
natural filtration $\F^{\delta} = \{\F^{k,\delta}\}_{k \geq 0}$ of the Markov
chain $x^{\delta}$.
The parameter $\delta$ represents a time increment. We define the piecewise
linear rescaled noise process by

\begin{equs} \label{e.WN}
W^{\delta}(t) = \sqrt{\delta} \; \sum_{j=0}^k \Gamma^{j,\delta} \;+\; \frac{t -
t_k}{\sqrt{\delta}} \; \Gamma^{k+1,\delta}
\qquad \text{for} \qquad
t_k \leq t < t_{k+1}.
\end{equs}
We now show that, as $\delta \to 0$, if the sequence of 
approximate drift functions $d^{\delta}(\cdot)$ converges in the appropriate
norm 
to a limiting drift $d(\cdot)$ and the sequence of rescaled noise process
$W^{\delta}$ converges to a Brownian motion 
then the sequence of piecewise linear interpolants 
$z^\delta$ defined by Equation \eqref{eqn:MCMCe} converges 
weakly to a diffusion process in $\h^s$. 
In order to state the general diffusion approximation Lemma
\ref{lem:diff_approx}, we introduce the following:

\begin{conditions} \label{con:diff:approx} There exists an integer $p \geq 1$ 
such that the sequence of Markov chains $x^\delta = \{x^{k,\delta}\}_{k \geq 0}$
satisfies
\begin{enumerate}
\item
{\bf Convergence of the drift:}
there exists a globally Lipschitz function $d:\h^s \to \h^s$ such that
\begin{equs} \label{e.drift.estim.cond}
\|d^\delta(x)-d(x)\|_s \;\lesssim\; \delta \cdot \big( 1+\|x\|^{p}_s \big)
\end{equs}
\item
{\bf Invariance principle:}
as $\delta$ tends to zero the sequence of processes $\{W^{\delta}\}_{\delta \in
(0,\frac12)}$ defined by Equation 
\eqref{e.WN} converges weakly in 
$C([0,T],\h^s)$ to a Brownian motion $W$ in $\h^s$
with covariance operator $C_s$.
\item
{\bf A priori bound:}
the following bound holds
\begin{equs} \label{cond.a.priori}
\sup_{\delta \in (0,\frac12)} \quad \Big\{\; \delta \cdot \EE\Big[ \sum_{k
\delta \leq T} \; \|x^{k,\delta}\|^p_s  \Big] \;\Big\} \;<\; \infty.
\end{equs}
\end{enumerate}
\end{conditions}
\begin{remark} The a-priori bound \eqref{cond.a.priori} can equivalently be
stated as   
\begin{equs}
\sup_{\delta \in (0,\frac12)} \; \Big\{ \EE\Big[ \int_0^T \|z^{\delta}(u)\|^p_s
\, du \Big] \Big\} < \infty.
\end{equs}
\end{remark}

\noindent
It is now proved that Conditions \ref{con:diff:approx} are sufficient to obtain
a diffusion approximation 
for the sequence of rescaled processes $z^{\delta}$ defined by
equation \eqref{eqn:MCMCe}, as $\delta$ tends to zero.
Contrary to more classical diffusion approximation for Markov processes results \cite{stroock1979multidimensional,Ethi:Kurt:86} based on infinitesimal generators, the next Lemma exploits specific structures which arise when the limiting
process has additive noise and, in particular, is based on 
exploiting preservation of weak convergence under continuous mappings,
together with an explicit construction of the noise process. This idea has 
previously appeared in the literature in, for example, the 
articles \cite{mattingly2011spde,PST12} in the context of MCMC
and the article \cite{kupferman2002long}, and the references therein, 
in the context of  the derivation of SDEs from ODEs with random data.

\begin{lem} \label{lem:diff_approx}
{\bf (General Diffusion Approximation for Markov chains)}\\
Consider a separable Hilbert space $\big( \h^s, \bra{\cdot, \cdot}_s \big)$ 
and a sequence of $\h^s$-valued Markov chains
$x^\delta = \{x^{k,\delta}\}_{k \geq 0}$
starting at a fixed position in the sense that 
$x^{0,\delta} = x_*$ for all $\delta \in (0,\frac12)$.
Suppose that the drift-martingale decompositions \eqref{eqn:markov_chain_seq} of
$x^{\delta}$ satisfy
Conditions \ref{con:diff:approx}.
Then the sequence of rescaled interpolants 
$z^\delta \in C([0,T],\h^s)$ defined by
equation \eqref{eqn:MCMCe}
converges weakly in $C([0,T],\h^s)$ 
to $z \in C([0,T],\h^s)$ given by the stochastic differential equation
\begin{equs} \label{eqn:spde:approx:lem}
    dz   &= d(z) \, dt + \sqrt{2\tau} dW\\
\end{equs}
with initial condition $z_0  = x_*$ and 
where $W$ is a Brownian motion in $\h^s$ with covariance $C_s$.
\end{lem}

\begin{proof}
For the sake of clarity, the proof of Lemma \ref{lem:diff_approx} is divided
into several steps.
\begin{itemize}

\item {\bf Integral equation representation.}\\
Notice that solutions of the $\h^s$-valued SDE \eqref{eqn:spde:approx:lem} are
nothing else than
solutions of the following integral equation,
\begin{equs} \label{e.integral.eq}
z(t) = x_* + \int_0^t \, d(z(u)) \, du + \sqrt{2 \tau} W(t) \qquad \forall t \in
(0,T),
\end{equs}
where $W$ is a Brownian motion in $\h^s$ with covariance operator
equal to $C_s$.
We thus introduce the It{\^o} map $\Theta: C([0,T],\h^s) \to C([0,T],\h^s)$ that
sends a function $W \in C([0,T],\h^s)$
to the unique solution of the integral equation 
\eqref{e.integral.eq}: solution of \eqref{eqn:spde:approx:lem} can be
represented as $\Theta(W)$ where 
$W$ is an $\h^s$-valued Brownian motion with covariance $C_s$.
As is described below, the function $\Theta$ is continuous if $C([0,T],\h^s)$
is topologized by the uniform norm $\|w\|_{C([0,T],\h^s)} \eqdef \sup\{
\|w(t)\|_{s} : t \in (0,T)\}$. 
It is crucial to notice that the rescaled process $z^{\delta}$, 
defined in Equation \eqref{eqn:MCMCe}, satisfies 
$z^{\delta} = \Theta(\widehat{W}^{\delta})$ with 
\begin{equs} \label{e.int.rep}
\widehat{W}^{\delta}(t) := W^{\delta}(t) + \frac{1}{\sqrt{2\tau}} 
\int_0^t [ d^{\delta}(\bar{z}^{\delta}(u))- d(z^{\delta}(u)) ]\,du.
\end{equs}
In Equation \eqref{e.int.rep}, the quantity $d^{\delta}$ is the approximate
drift defined 
in Equation \eqref{eqn:approx:drift:noise} and $\bar{z}^{\delta}$ is the
rescaled piecewise 
constant interpolant of $\{x^{k,\delta}\}_{k \geq 0}$ defined as
\begin{equs}  \label{e.piecewise.cst}
\bar{z}^{\delta}(t) = x^{k,\delta} 
\qquad \text{for} \qquad t_k \leq t < t_{k+1}.
\end{equs}
The proof follows from a continuous mapping argument (see below) once it is
proven
that $\widehat{W}^{\delta}$ converges weakly in $C([0,T],\h^s)$ to $W$.

\item 
{\bf The It{\^o} map $\Theta$ is continuous}\\
It can be proved that $\Theta$ is continuous as a mapping 
from $\Big(C([0,T],\h^s), \|\cdot\|_{C([0,T],\h^s)} \Big)$ to itself.
The usual Picard's iteration proof of the Cauchy-Lipschitz theorem of ODEs may
be employed: 
see \cite{mattingly2011spde}.

\item {\bf The sequence of processes $\widehat{W}^{\delta}$ converges weakly to
$W$}\\
The process  $\widehat{W}^{\delta}(t)$ is defined by
$\widehat{W}^{\delta}(t) = W^{\delta}(t) + \frac{1}{\sqrt{2\tau}} \int_0^t [
d^{\delta}(\bar{z}^{\delta}(u))- d(z^{\delta}(u)) ]\,du$ 
and Conditions \ref{con:diff:approx}  state that $W^{\delta}$ converges weakly
to $W$ in $C([0,T], \h^s)$.
Consequently, to prove that $\widehat{W}^{\delta}(t)$ converges weakly to $W$ in
$C([0,T], \h^s)$,
it suffices (Slutsky's lemma) to verify that the sequences of processes
\begin{equs} \label{e.conv.proba.0}
(\omega,t) \mapsto \int_0^t \big[ d^{\delta}(\bar{z}^{\delta}(u))-
d(z^{\delta}(u)) \big] \,du 
\end{equs}
converges to zero in probability with respect to the supremum norm in
$C([0,T],\h^s)$. 
By Markov's inequality, it is enough to check that
$\EE \big[ \int_0^T \; \| d^{\delta}(\bar{z}^{\delta}(u))-
d(z^{\delta}(u)) \|_s \; du \big]$ converges to zero as $\delta$ goes to zero.
Conditions \ref{con:diff:approx} states that there exists an integer $p \geq 1$
such that
$\|d^{\delta}(x)-d(x)\| \lesssim \delta \cdot (1+\|x\|_s^{p})$ so that for any
$t_k \leq u < t_{k+1}$
we have
\begin{equs} \label{e.bound.1}
\Big\|d^{\delta}(\bar{z}^{\delta}(u))- d(\bar{z}^{\delta}(u)) \Big\|_s 
\lesssim \delta \; \big( 1+\|\bar{z}^{\delta}(u)\|_s^{p} \big)
= \delta \; \big( 1+\|x^{k,\delta}\|_s^{p} \big).
\end{equs}
Conditions \ref{con:diff:approx} states that $d(\cdot)$ is globally Lipschitz on
$\h^s$. Therefore, Lemma \ref{lem:size:prop} shows that
\begin{equs}\label{e.bound.2}
\EE \|d(\bar{z}^{\delta}(u))- d(z^{\delta}(u))\|_s 
\;\lesssim\; \EE \|x^{k+1,\delta}-x^{k,\delta}\|_s 
\;\lesssim\; \delta^{\frac12} (1+\|x^{k,\delta}\|_s).
\end{equs}
From estimates \eqref{e.bound.1} and \eqref{e.bound.2} it follows that 
$\| d^{\delta}(\bar{z}^{\delta}(u))- d(z^{\delta}(u)) \|_s \lesssim
\delta^{\frac12} (1+\|x^{k,\delta}\|^{p}_s)$.
Consequently
\begin{equs} \label{e.final.bound}
\EE \Big[ \int_0^T \; \| d^{\delta}(\bar{z}^{\delta}(u))- d(z^{\delta}(u)) \|_s
\; du \Big]
\;\lesssim\;
\delta^{\frac32} \; \sum_{k \delta < T} \EE\Big[ 1+\|x^{k,\delta}\|^{p}_s
\Big]. 
\end{equs}
The a-priori bound of Conditions \ref{con:diff:approx} shows that this last
quantity converges to zero 
as $\delta$ converges to zero, which finishes the proof of Equation
\eqref{e.conv.proba.0}. 
This concludes the proof of $\widehat{W}^{\delta}(t) \longweak W$.

\item{\bf Continuous mapping argument.}\\
It has been proved that $\Theta$ is continuous as a mapping 
from $\Big(C([0,T],\h^s), \|\cdot\|_{C([0,T],\h^s)} \Big)$ to itself.
The solutions of the $\h^s$-valued SDE \eqref{eqn:spde:approx:lem} 
can be expressed as $\Theta(W)$ while the 
rescaled continuous interpolate $z^{\delta}$ also reads 
$z^{\delta} = \Theta(\widehat{W}^{\delta})$. 
Since $\widehat{W}^{\delta}$ converges weakly in  
$\Big(C([0,T],\h^s), \|\cdot\|_{C([0,T],\h^s)} \Big)$ to $W$ as $\delta$ tends
to zero, 
the continuous mapping
theorem ensures that $z^{\delta}$ converges weakly in  
$\Big(C([0,T],\h^s), \|\cdot\|_{C([0,T],\h^s)} \Big)$ to the 
solution $\Theta(W)$ of the $\h^s$-valued 
SDE \eqref{eqn:spde:approx:lem}. 
This ends the proof of Lemma \ref{lem:diff_approx}.

\end{itemize}
\end{proof}

\noindent
In order to establish Theorem \ref{thm:main} as a consequence of the general 
diffusion approximation Lemma \ref{lem:diff_approx}, it suffices to verify 
that if Assumptions \ref{ass:1} hold then Conditions \ref{con:diff:approx} 
are satisfied by the Markov chain 
$x^{\delta}$ defined in section \ref{sec:algorithm}. 
In section \ref{sec:drift} we prove 
the following quantitative version of the approximation the function 
$d^{\delta}(\cdot)$ by the function $d(\cdot)$ 
where $d(x) = -\Big( x + C \nabla \Psi(x)\Big)$.

\begin{lem} {\bf (Drift estimate)} \label{lem.drift.estim}\\
Let Assumptions \ref{ass:1} hold and let $p \geq 1$ be an integer. 
Then the following estimate is satisfied,
\begin{equs} \label{e.drift.approx.pos}
\|d^{\delta}(x) - d(x)\|_s^p 
\; \lesssim \;
\delta^{\frac{p}{2}} (1+\|x\|_s^{2p}). 
\end{equs}
Moreover, the approximate drift $d^{\delta}$ is linearly bounded 
in the sense that
\begin{equs} \label{e.sublinear}
\| d^{\delta}(x) \|_{s} 
\; \lesssim \; 
1+\|x\|_s.
\end{equs}
\end{lem}

\noindent
It follows from Lemma \eqref{lem.drift.estim} 
that Equation \eqref{e.drift.estim.cond} of Conditions \ref{con:diff:approx} 
is satisfied as soon as Assumptions \ref{ass:1} hold.
The invariance principle of Conditions \ref{con:diff:approx} follows from the
next lemma. It is
proved in section \ref{sec:invariance:principle}.

\begin{lem} {\bf (Invariance Principle)} \label{lem:bweakconv}\\
Let Assumptions \ref{ass:1} hold. Then the rescaled noise process 
$W^{\delta}(t)$ defined in equation \eqref{e.WN} satisfies
\begin{equs}
 W^{\delta} \quad \longweak \quad W
\end{equs}
where $\longweak$ denotes weak convergence in $C([0,T],\h^s)$, and $W$ is
a $\h^s$-valued Brownian motion with covariance operator $C_s$.
\end{lem}

\noindent
In section \ref{sec:apriori:bound} it is proved that the following a priori
bound is satisfied,
\begin{lem} \label{lem:apriori:bound} {\bf (A priori bound)}\\
Consider a fixed time horizon $T > 0$ and an integer $p \geq 1$. 
Under Assumptions \ref{ass:1} the following bound holds,
\begin{equs} \label{e.a.priori}
\sup \quad \Big\{ \delta \cdot \EE\Big[ \sum_{k \delta \leq T} \;
\|x^{k,\delta}\|^p_s  \Big] \; : \delta \in (0,\frac12) \Big\} \;<\; \infty.
\end{equs}
\end{lem}

\noindent
In conclusion, Lemmas \ref{lem.drift.estim} and \ref{lem:bweakconv} and
\ref{lem:apriori:bound} together show 
that Conditions \ref{con:diff:approx} are consequences of Assumptions
\ref{ass:1}. Therefore, under Assumptions 
\ref{ass:1}, the general diffusion approximation Lemma \ref{lem:diff_approx} can
be applied:
this concludes the proof of Theorem \ref{thm:main}.

\section{Key Estimates}\label{sec:4}
This section assembles various results
which are used in the previous section. Some of the technical proofs are
deferred to the appendix.

\subsection{Acceptance Probability Asymptotics} \label{sec:acceptance:prob}
This section describes a first order expansion of the acceptance probability.
The approximation
\begin{equs} \label{e.accept.prob.approx}
\alpha^{\delta}(x,\xi) \approx \bar{\alpha}^{\delta}(x,\xi)
\quad \text{where} \quad
\bar{\alpha}^{\delta}(x,\xi) = 1 - \sqrt{\frac{2\delta}{\tau}} \bra{\nabla
\Psi(x), \xi} \indic_{\{ \bra{\nabla \Psi(x), \xi} > 0 \}}
\end{equs}
is valid for $\delta \ll 1$. 
The quantity $\bar{\alpha}^{\delta}$ has the advantage over $\alpha^{\delta}$ 
of being very simple to analyse: explicit computations are available. 
This will be exploited in section \ref{sec:drift}.
The quality of the approximation \eqref{e.accept.prob.approx} 
is rigorously quantified in the next lemma.
\begin{lem} {\bf (Acceptance probability estimate)}
\label{lem.acc.prob.estim}\\
Let Assumptions \ref{ass:1} hold. For any integer $p \geq 1$ 
the quantity $\bar{\alpha}^{\delta}(x,\xi)$ satisfies
\begin{equs} \label{e.acc.asymp}
\EE_x\big[ |\alpha^{\delta}(x,\xi) - \bar{\alpha}^{\delta}(x,\xi)|^p]
\;\lesssim\; \delta^{p} \; (1+\|x\|_s^{2p}). 
\end{equs}
\end{lem}
\begin{proof}
See Appendix \ref{app:sec4calc}.
\end{proof}
\noindent
Recall the local mean acceptance $\alpha^{\delta}(x)$ defined in Equation \eqref{e.mean.loc.acc}.
Define the approximate local mean acceptance probability by
$\bar{\alpha}^{\delta}(x) \eqdef \EE_x[ \bar{\alpha}^{\delta}(x,\xi)]$.
One can use Lemma \ref{lem.acc.prob.estim} to approximate the local mean
acceptance probability $\alpha^{\delta}(x)$.
\begin{corollary} \label{cor.alpha.bar}
Let Assumptions \ref{ass:1} hold. 
For any integer $p \geq 1$ the following estimates hold,
\begin{equs} 
&\big| \alpha^{\delta}(x) - \bar{\alpha}^{\delta}(x) \big| \;\lesssim\; \delta
\, (1+\|x\|_s^{2}) \label{e.bound.alpha.apha.bar.1}\\
&\EE_x \Big[ \; \big|\alpha^{\delta}(x,\xi)  \;-\;  1 \big|^p \Big] \;\lesssim\;
\delta^{\frac{p}{2}} \, (1+\|x\|^p_s) \label{e.bound.alpha.apha.bar.2}
\end{equs}
\end{corollary}
\begin{proof}
See Appendix \ref{app:sec4calc}.
\end{proof}

\subsection{Drift Estimates} \label{sec:drift}
Explicit computations are available for the quantity 
$\bar{\alpha}^{\delta}$.
We will use these results, together with quantification of the
error committed 
in replacing $\alpha^{\delta}$ by $\bar{\alpha}^{\delta}$, 
to estimate the mean drift (in this section) and the diffusion term (in the next
section).

\begin{lem} \label{lem.alpha.bar}
For any $x \in \h^s$ the approximate acceptance probability $\bar{\alpha}^{\delta}(x,\xi)$ satisfies
\begin{equs}
\sqrt{\frac{2\tau}{\delta}} \; \EE_x \Big[ \bar{\alpha}^{\delta}(x,\xi) \cdot
\xi \Big] 
\quad=\quad -C \nabla \Psi(x)
.
\end{equs}
\end{lem}

\begin{proof}
Let $u = \sqrt{\frac{2\tau}{\delta}} \; \EE_x \Big[ \bar{\alpha}^{\delta}(x,\xi)
\cdot \xi \Big] \in \h^s$. To prove 
the lemma it suffices to verify that for all $v \in \h^{-s}$ we have
$
\bra{u,v} \;=\; -\bra{C \nabla \Psi(x), v}$.
To this end, use the decomposition $v = \alpha \nabla \Psi(x) + w$ 
where $\alpha \in \RR$ and $w \in \h^{-s}$ satisfies $\bra{C \nabla \Psi(x), w}
= 0$.
Since $\xi \dist \Normal(0,C)$ the two Gaussian random variables
$Z_{\Psi} \eqdef \bra{\nabla \Psi(x), \xi}$ 
and 
$Z_{w} \eqdef \bra{w, \xi}$ are independent: indeed, $(Z_{\Psi}, Z_{w})$ is a Gaussian vector in $\RR^2$
with $\text{Cov}(Z_{\Psi}, Z_{w}) = 0$.
It thus follows that
\begin{equs}
\bra{u,v} 
&= -2 \; \bra{ \EE_x \big[\bra{\nabla \Psi(x), \xi} 1_{\{ \bra{\nabla \Psi(x),
\xi} > 0 \}} \cdot \xi \big] \;,\; \alpha \nabla \Psi(x) + w} \\
&= -2 \; \EE_x\Big[ \alpha Z_{\Psi}^2 1_{\{ Z_{\Psi} > 0\}} \;+\; Z_{w} \,
Z_{\Psi} 1_{\{ Z_{\Psi} > 0\}}\Big] \\
&= -2\alpha \; \EE_x\Big[ Z_{\Psi}^2 1_{\{ Z_{\Psi} > 0\}} \Big]  \;=\; -\alpha
\EE_x\Big[ Z_{\Psi}^2 \Big]\\
&= -\alpha \bra{C \nabla \Psi(x), \nabla \Psi(x)} \;=\; \bra{-C \nabla \Psi(x),
\alpha \nabla \Psi(x) + w} \\
&= -\bra{C \nabla \Psi(x), v},
\end{equs}
which concludes the proof of Lemma \ref{lem.alpha.bar}.
\end{proof}

\noindent
We now use this explicit computation to give a proof of the drift estimate Lemma
\ref{lem.drift.estim}.

\begin{proof}[Proof of Lemma \ref{lem.drift.estim}]
The function $d^{\delta}$ defined by Equation \eqref{eqn:approx:drift:noise} can
also be expressed as
\begin{equs} \label{e.ddelta.splitting}
d^{\delta}(x) 
\;=\; \Big\{ \frac{(1-2\delta)^{\frac12} - 1}{\delta }\alpha^{\delta}(x) \, x
\Big\}
+
\Big\{ \sqrt{\frac{2\tau}{\delta}} \; \EE_x[\alpha^{\delta}(x,\xi) \, \xi]
\Big\}
\;=\; B_1 + B_2,
\end{equs}
where the mean local acceptance probability $\alpha^{\delta}(x)$ has been
defined in Equation \eqref{e.mean.loc.acc} and 
the two terms $B_1$ and $B_2$ are studied below.
To prove Equation \eqref{e.drift.approx.pos}, it suffices to establish that
\begin{equs} \label{e.bound.Bi}
\|B_1+x\|_s^p \lesssim \delta^{\frac{p}{2}} (1+\|x\|_s^{2p})
\qquad \text{and} \qquad 
\|B_2 + C\nabla \Psi(x)\|_s^p \lesssim  \delta^{\frac{p}{2}} (1+\|x\|_s^{2p}).
\end{equs}
We now establish these two bounds.
\begin{itemize}
\item 
Lemma \ref{lem.acc.prob.estim} and Corollary \ref{cor.alpha.bar} show that
\begin{equs} \label{e.drift.approx.1}
\|B_1 + x \|_s^p
&= \Big\{ \frac{(1-2\delta)^{\frac12} - 1}{\delta}\alpha^{\delta}(x) + 1
\Big\}^p \; \|x\|^p_s \\
&\lesssim \Big\{ \big| \frac{(1-2\delta)^{\frac12} - 1}{\delta} - 1 \big|^p +  
\big| \alpha^{\delta}(x) - 1 \big|^p \Big\}\; \|x\|^p_s \\
&\lesssim \Big\{\delta^p +  \delta^\frac{p}{2} (1+\|x\|_s^p) \Big\}\; \|x\|^p_s 
\;\lesssim\; \delta^\frac{p}{2} (1+\|x\|_s^{2p}).
\end{equs}

\item
Lemma \ref{lem.acc.prob.estim} shows that
\begin{equs} \label{e.drift.approx.2}
\|B_2 + C\nabla \Psi(x)\|_s^p
&= \big\| \sqrt{\frac{2\tau}{\delta}} \, \EE_x[\alpha^{\delta}(x,\xi) \, \xi] +
C\nabla \Psi(x)  \big\|_s^p \\
&\lesssim  \delta^{-\frac{p}{2}} \big\| \EE_x[\{\alpha^{\delta}(x,\xi) -
\bar{\alpha}^{\delta}(x,\xi) \} \, \xi] \big\|_s^p\\
&\qquad+ \big\|  \underbrace{ \sqrt{\frac{2\tau}{\delta}} \,
\EE_x[\bar{\alpha}^{\delta}(x,\xi)  \, \xi] + C\nabla \Psi(x)}_{=0}  \big\|_s^p.
\end{equs}
By Lemma \ref{lem.alpha.bar}, the second term on the right hand equals to
zero. Consequently, the Cauchy-Schwarz inequality 
implies that
\begin{equs}
\|B_2 + C\nabla \Psi(x)\|_s^p
&\lesssim  \delta^{-\frac{p}{2}} \EE_x[ \big|\alpha^{\delta}(x,\xi) -
\bar{\alpha}^{\delta}(x,\xi)\big|^2]^{\frac{p}{2}}\\
&\lesssim  \delta^{-\frac{p}{2}} \Big( \delta^2
(1+\|x\|_s^{4})\Big)^{\frac{p}{2}} \;\lesssim\; \delta^{\frac{p}{2}}
(1+\|x\|_s^{2p}).
\end{equs}
\end{itemize}
Estimates \eqref{e.drift.approx.1} and \eqref{e.drift.approx.2} give Equation
\eqref{e.bound.Bi}.
To complete the proof we establish the bound \eqref{e.sublinear}. The expression
\eqref{e.ddelta.splitting} 
shows that it suffices to verify
$\delta^{-\frac12} \; \EE_x[\alpha^{\delta}(x,\xi) \, \xi] 
\quad \lesssim \quad 1+\|x\|_s$.
To this end, we use Lemma \ref{lem.alpha.bar} and Corollary \ref{cor.alpha.bar}.
By the Cauchy-Schwarz inequality,
\begin{equs}
\Big\| \delta^{-\frac12} \; \EE_x \Big[\alpha^{\delta}(x,\xi) \cdot
\xi \Big] \Big\|_s
&\;=\; 
\Big\| \delta^{-\frac12} \; \EE_x \Big[[\alpha^{\delta}(x,\xi) - 1] \cdot \xi \Big] \Big\|_s\\
&\;\lesssim\; \delta^{-\frac12} \;  \EE_x \Big[ (\alpha^{\delta}(x,\xi) - 1)^2
\Big]^{\frac12}
\;\lesssim\; 1 + \|x\|_s,
\end{equs}
which concludes the proof of Lemma \ref{lem.drift.estim}.
\end{proof}

\subsection{Noise Estimates} \label{sec:noise}
In this section we estimate the error in the approximation $\Gamma^{k,\delta}
\approx \Normal(0,C_s)$.
To this end, let us introduce the covariance operator $D^{\delta}(x)=\EE \Big[ \Gamma^{k,\delta} \otimes_{\h^s} \Gamma^{k,\delta} \;|
x^{k,\delta}=x \Big]$ of the
martingale difference $\Gamma^{\delta}$.
For any $x, u,v \in \h^s$ the operator $D^\delta(x)$ satisfies
\begin{equs}
\EE\Big[\bra{\Gamma^{k,\delta},u}_s \bra{\Gamma^{k,\delta},v}_s
\;|x^{k,\delta}=x \Big]  \;=\; \bra{u, D^\delta(x) v}_s.
\end{equs}
The next lemma gives a quantitative version of the approximation of $D^\delta(x)$ 
by the operator $C_s$.

\begin{lem} {\bf (Noise estimates)} \label{lem.noise.estim}\\
Let Assumptions \ref{ass:1} hold. For any pair of indices $i,j \geq 1$, 
the martingale difference term $\Gamma^\delta(x,\xi)$ satisfies 
\begin{equs} 
&| \bra{\pphi_i, D^{\delta}(x) \; \pphi_j}_s \;-\; \bra{\pphi_i,C_s \;
\pphi_j}_s | 
\;\lesssim\; \delta^{\frac18} \cdot \big( 1+ \|x\|_s \big) 
\label{e.noise.approx.1}\\
&|\tr_{\h^s}\big( D^{\delta}(x) \big) \;-\; \tr_{\h^s}\big( C_s \big) |
\;\lesssim\; \delta^{\frac18} \cdot \big( 1+ \|x\|_s^2 \big).  
\label{e.noise.approx.2}
\end{equs}
with $\{\pphi_j = j^{-s} \phi_j\}_{j \geq 0}$ is an orthonormal basis of $\h^s$.
\end{lem}
\begin{proof}
See Appendix \ref{app:sec4calc}.
\end{proof}

\subsection{A Priori Bound}
 \label{sec:apriori:bound}

Now we have all the ingredients for the proof of the a priori bound presented in
Lemma \ref{lem:apriori:bound} which states that
the rescaled process $z^{\delta}$ given by Equation \eqref{eqn:MCMCe} does not
blow up in finite time. 
\begin{proof}[Proof Lemma \ref{lem:apriori:bound}]
Without loss of generality, assume that $p = 2n$ for some positive integer $n
\geq 1$.
We now prove that there exist constants $\alpha_1, \alpha_2, \alpha_3 > 0$
satisfying
\begin{equs} \label{e.exp.decay}
\EE[\|x^{k,\delta}\|_s^{2n}] \leq (\alpha_1 + \alpha_2 k \, \delta) e^{\alpha_3
k \, \delta}.
\end{equs}
Lemma \ref{lem:apriori:bound} is a straightforward consequence of Equation
\ref{e.exp.decay} 
since this implies that
\begin{equs}
\delta \; \sum_{k \delta < T} \EE[\|x^{k,\delta}\|_s^{2n}] 
\;\leq\; \delta \; \sum_{k \delta < T}(\alpha_1 + \alpha_2 k \, \delta)
e^{\alpha_3 k \, \delta}
\;\asymp\; \int_0^T \; (\alpha_1 + \alpha_2 \;t) \; e^{\alpha_3 \; t} < \infty.
\end{equs}
For notational convenience, let us define $V^{k,\delta} =
\EE\big[\|x^{k,\delta}\|_s^{2n} \big]$.
To prove Equation \eqref{e.exp.decay}, it suffices to establish that
\begin{equs} \label{e.recc}
V^{k+1, \delta}-V^{k, \delta} \leq K \, \delta \cdot  \big( 1+V^{k,\delta}\big),
\end{equs}
where $K>0$ is constant independent from $\delta \in (0,\frac12)$.
Indeed, iterating inequality \eqref{e.recc} leads to the bound
\eqref{e.exp.decay},
for some computable constants $\alpha_1, \alpha_2, \alpha_3 > 0$.
The definition of $V^k$ shows that
\begin{equs} \label{e.binom}
V^{k+1,\delta} - V^{k,\delta}
&= \EE \big[ \| x^{k,\delta} + (x^{k+1,\delta}-x^{k,\delta})\|_s^{2n} -
\|x^{k,\delta}\|_s^{2n} \big] \\
&= \EE \big[ \Big\{ \|x^{k,\delta} \|_s^2 + \|x^{k+1,\delta}-x^{k,\delta}\|_s^2 \\
&\qquad + 2\bra{x^{k,\delta}, x^{k+1,\delta}-x^{k,\delta}}_s\Big\}^n -
\|x^{k,\delta}\|_s^{2n}  \big]
\end{equs}
where the increment $x^{k+1,\delta}-x^{k,\delta}$ is given by
\begin{equs} \label{e.incr}
x^{k+1,\delta}-x^{k,\delta} \;=\; \gamma^{k,\delta} \; \Big(
(1-2\delta)^{\frac12} - 1\Big) x^{k,\delta} 
+ \sqrt{2\delta} \; \gamma^{k,\delta} \xi^k.
\end{equs}
To bound the right-hand-side of Equation \eqref{e.binom}, we use a binomial
expansion and control 
each term. To this end, we establish the following estimate: for all integers
$i,j,k \geq 0$ satisfying 
$i+j+k = n$ and $(i,j,k) \neq (n,0,0)$ 
the following inequality holds,
\begin{equs} \label{e.bound.binom}
\EE\Big[&\Big(\|x^{k,\delta} \|_s^2\Big)^i 
\times \Big(\|x^{k+1,\delta}-x^{k,\delta}\|_s^2 \Big)^j \\
&\qquad \times \Big( \bra{x^{k,\delta},
x^{k+1,\delta}-x^{k,\delta}}_s \Big)^k \Big]
\; \lesssim \; \delta  \, (1+V^{k,\delta}).
\end{equs}
To prove Equation \eqref{e.bound.binom}, we separate two different cases.
\begin{itemize}
\item
Let us suppose $(i,j,k) = (n-1,0,1)$. Lemma \ref{lem.drift.estim} states that 
the approximate drift has a linearly bounded growth so that
\begin{equs}
\Big\| \EE\big[ x^{k+1,\delta} - x^{k,\delta} \;| x^{k,\delta} \big] \Big\|_s
\quad = \quad \delta \; \| d^{\delta}(x^{k,\delta})\|_s 
\quad \lesssim \quad \delta \; ( 1+\|x^{k,\delta}\|_s ).
\end{equs}
Consequently, we have
\begin{equs}
\EE\Big[\Big(\|x^{k,\delta} \|_s^2\Big)^{n-1} \; \bra{x^{k,\delta},
x^{k+1,\delta}-x^{k,\delta}}_s  \Big]
&\lesssim \EE\Big[ \|x^{k,\delta} \|_s^{2(n-1)} \; \|x^{k,\delta} \|_s \; \Big(
\delta \; (1+\|x^{k,\delta}\|_s \Big) \Big] \\
&\lesssim \delta (1 + V^{k,\delta}).
\end{equs}
This proves Equation \eqref{e.bound.binom} in the case $(i,j,k) = (n-1,0,1)$.

\item 
Let us suppose $(i,j,k) \not \in \Big\{ (n,0,0), (n-1,0,1) \Big\}$. Because for
any integer $p \geq 1$,
\begin{equs}
\EE_x\Big[ \|x^{k+1,\delta}-x^{k,\delta}\|_s^p \Big]^{\frac{1}{p}} \lesssim
\delta^{\frac12} \; (1+\|x\|_s)
\end{equs}
it follows from the Cauchy-Schwarz inequality that
\begin{equs}
\EE\Big[\Big(\|x^{k,\delta} \|_s^2\Big)^i \; \Big(
\|x^{k+1,\delta}-x^{k,\delta}\|_s^2 \Big)^j \; \Big( \bra{x^{k,\delta},
x^{k+1,\delta}-x^{k,\delta}}_s \Big)^k \Big]
\; \lesssim \; \delta^{j + \frac{k}{2}} \; (1+ V^{k,\delta}).
\end{equs}
Since we have supposed that $(i,j,k) \not \in \Big\{ (n,0,0), (n-1,0,1) \Big\}$
and $i+j+k = n$, it follows that $j + \frac{k}{2} \geq 1$.
This concludes the proof of Equation \eqref{e.bound.binom},
\end{itemize}
The binomial expansion of Equation \eqref{e.binom} and the bound
\eqref{e.bound.binom} show that Equation \eqref{e.recc} holds. This concludes the proof of Lemma
\ref{lem:apriori:bound}.
\end{proof}

\subsection{Invariance Principle} 
\label{sec:invariance:principle}

Combining the noise estimates of Lemma \ref{lem.noise.estim} and the a priori
bound of Lemma \ref{lem:apriori:bound},
we show that under Assumptions \ref{ass:1} the sequence of rescaled noise
processes defined in 
Equation \ref{e.WN} converges weakly to a Brownian motion. This is the content
of Lemma \ref{lem:bweakconv} 
whose proof is now presented.
\begin{proof}[Proof of Lemma \ref{lem:bweakconv}]
As described in \cite{Berg:86} [Proposition $5.1$], in order to prove that 
$W^{\delta}$ converges weakly to $W$ in $C([0,T], \h^s)$ it suffices to prove
that 
for any $t \in [0,T]$ and any pair of indices $i,j \geq 0$ the following three
limits hold in probability,

\begin{eqnarray}
&\lim_{\delta \rightarrow 0}& \delta \, \sum_{k \delta < t} 
\EE\Big[\|\Gamma^{k,\delta}\|_s^2 \;| x^{k,\delta} \Big] \; = \; t \cdot
\tr_{\h^s}(C_s) 
\label{cond_1}\\
&\lim_{\delta \rightarrow 0}& \delta \, \sum_{k\delta < t} 
\EE\Big[ \bra{ \Gamma^{k,\delta}, \pphi_i}_s \bra{ \Gamma^{k,\delta}, \pphi_j
}_s \; | x^{k,\delta}\Big] 
\; = \; t \,\bra{ \pphi_i, C_s \pphi_j }_s 
\label{cond_2} \\
&\lim_{\delta \rightarrow 0}& \delta \,  \sum_{k\delta < T} 
\EE\Big[\| \Gamma^{k,\delta}\|_s^2 \; \indic_{ \{\| \Gamma^{k,\delta} \|_s^2 
\geq \delta^{-1} \, \epsilon \}} \; |x^{k,\delta} \Big] \; = \; 0  \qquad
\forall \eps > 0
\label{cond_3}.
\end{eqnarray}
We now check that these three conditions are indeed satisfied.
\begin{itemize}
\item Condition \eqref{cond_1}:
since $\EE\Big[ \|\Gamma^{k,\delta}\|_s^2 \;| x^{k,\delta}\Big] =
\tr_{\h^s}(D^{\delta}(x^{k,\delta}))$,
Lemma \ref{lem.noise.estim} shows that
\begin{equs}
\EE\Big[ \|\Gamma^{k,\delta}\|_s^2 \;| x^{k,\delta}\Big]
\;=\; \tr_{\h^s}(C_s) + \err_1^{\delta}(x^{k,\delta})
\end{equs}
where the error term $\err_1^{\delta}$ satisfies $| \err_1^{\delta}(x) |
\lesssim \delta^{\frac18} \; (1+\|x\|_s^2)$.
Consequently, to prove condition \eqref{cond_1} it suffices to establish that 
\begin{equs}
\lim_{\delta \to 0} \; \EE \Big[ \big| \delta \, \sum_{k \delta < T}
\err_1^{\delta}(x^{k,\delta}) \big| \Big] = 0 .
\end{equs}
We have $\EE \big[ \big| \delta \, \sum_{k \delta < T}
\err_1^{\delta}(x^{k,\delta}) \big| \big] 
\lesssim \delta^{\frac18} \Big\{ \delta \cdot \EE \Big[ \sum_{k \delta < T} \;
(1+\|x^{k,\delta}\|_s^2) \Big] \Big\}$
and the a priori bound presented in Lemma \ref{lem:apriori:bound} shows that 
\begin{equs}
\sup_{\delta \in (0,\frac12)} \quad \Big\{ \delta \cdot \EE \Big[ \sum_{k \delta
< T} \; (1+\|x^{k,\delta}\|_s^2) \Big] \Big\} < \infty.
\end{equs}
Consequently $\lim_{\delta \to 0} \; \EE \big[ \big| \delta \, \sum_{k \delta <
T} \err_1^{\delta}(x^{k,\delta}) \big| \big] = 0$, 
and the conclusion follows.

\item Condition \eqref{cond_2}:
Lemma \ref{lem.noise.estim} states that
\begin{equs}
\EE_k \Big[ \bra{ \Gamma^{k,\delta}, \pphi_i}_s \bra{ \Gamma^{k,\delta}, \pphi_j
}_s \Big] 
\;=\; \bra{ \pphi_i, C_s \pphi_j }_s + \err_2^{\delta}(x^{k,\delta})
\end{equs}
where the error term $\err_2^{\delta}$ satisfies $| \err_2^{\delta}(x) |
\lesssim \delta^{\frac18} \; (1+\|x\|_s)$.
The exact same approach as the proof of Condition \eqref{cond_1} gives the
conclusion.

\item Condition \eqref{cond_3}: from the Cauchy-Schwarz and Markov's inequalities it
follows that
\begin{equs}
\EE\Big[\| \Gamma^{k,\delta}\|_s^2 \; \indic_{ \{\| \Gamma^{k,\delta} \|_s^2
\geq \delta^{-1} \, \epsilon \}} \Big]
&\leq \EE\Big[\| \Gamma^{k,\delta}\|_s^4 \Big]^{\frac12} \cdot \PP\Big[ \|
\Gamma^{k,\delta} \|_s^2 \geq \delta^{-1} \, \epsilon \Big]^{\frac12}\\
&\leq \EE\Big[\| \Gamma^{k,\delta}\|_s^4 \Big]^{\frac12} \cdot \Big\{
\frac{\EE\big[\| \Gamma^{k,\delta}\|_s^4 \big]}{(\delta^{-1} \,
\epsilon)^2}\Big\}^{\frac12}\\
&\leq \frac{\delta}{\eps} \cdot \EE\Big[\| \Gamma^{k,\delta}\|_s^4
\Big].
\end{equs}
Lemma \ref{lem.drift.estim} readily shows that 
$\EE\| \Gamma^{k,\delta}\|_s^4 \lesssim 1+\|x\|_s^4$
Consequently we have
\begin{equs}
\EE\Big[ \Big| \delta \,  \sum_{k\delta < T} \EE\Big[\| \Gamma^{k,\delta}\|_s^2
\; \indic_{ \{\| \Gamma^{k,\delta} \|_s^2 
\geq \delta^{-1} \, \epsilon \}} \; |x^{k,\delta} \Big] \;\Big| \Big] 
&\leq \frac{\delta}{\eps} \times \Big\{  \delta \cdot \EE \Big[ \sum_{k
\delta < T} \; (1+\|x^{k,\delta}\|_s^4) \Big] \Big\}
\end{equs}
and the conclusion again follows from the a priori bound Lemma
\ref{lem:apriori:bound}.
\end{itemize}
\end{proof}

\section{Quadratic Variation}
\label{sec:5}

As discussed in the introduction, the SPDE 
\eqref{eqn:limit:SPDEG}, and the Metropolis-Hastings
algorithm pCN which approximates it for small $\delta$,
do not satisfy the smoothing property
and so almost sure properties of the limit measure $\pi^\tau$
are not necessarily seen at finite time. To illustrate this point, 
we introduce in this section a functional $V:\h \to \RR$ that is well defined on a dense subset of $\h$ and such that $V(X)$ is $\pi^{\tau}$-almost 
surely well defined and satisfies $\PP\big( V(X) = 1\big) = \tau$ 
for $X \dist \pi^{\tau}$. The quantity $V$ corresponds to the usual quadratic variation if $\pi_0$ is the Wiener measure.
We show that the quadratic variation like quantity $V(x^{k,\tau})$ of a pCN 
Markov chain converges
as $k \to \infty$ to the almost sure quantity $\tau$.
We then prove that piecewise linear interpolation
of this quantity solves, in the small $\delta$ limit,
a linear ODE (the
``fluid limit'') whose globally attractive stable
state is the almost sure quantity $\tau$. This quantifies
the manner in which the pCN method approaches
statistical equilibrium.

\subsection{Definition and Properties}
Under Assumptions \ref{ass:1}, the Karhunen-Lo{\'e}ve expansion shows 
that $\pi_0$-almost every $x \in \h$ satisfies 
$$\lim_{N \to \infty} \; N^{-1} \, \sum_{j=1}^N
\frac{\bra{x,\phi_j}^2}{\lambda_j^2} = 1.$$ 
This motivates the definition of the quadratic variation like quantities
\begin{equs}
V_-(x) \eqdef \liminf_{N \to \infty} \; N^{-1} \, \sum_{j=1}^N
\frac{\bra{x,\phi_j}^2}{\lambda_j^2}
\quad \text{and} \quad
V_+(x) \eqdef \limsup_{N \to \infty} \; N^{-1} \, \sum_{j=1}^N
\frac{\bra{x,\phi_j}^2}{\lambda_j^2}.
\end{equs}
When these two quantities are equal the vector $x \in \h$ is said 
to possess a {\em quadratic variation} $V(x)$ defined as $V(x) = V_-(x) =
V_+(x)$. 
Consequently, $\pi_0$-almost every $x \in \h$ possesses a quadratic variation
$V(x)=1$. 
It is a straightforward consequence that
$\pi_0^{\tau}$-almost every and $\pi^{\tau}$-almost every
$x \in \h$ possesses a quadratic variation $V(x)=\tau$.
Strictly speaking this only coincides with quadratic
variation when $C$ is the covariance of a (possibly
conditioned) Brownian motion; however we use the
terminology more generally in this section.
The next lemma proves that the quadratic variation $V(\cdot)$ behaves as it
should do with respect to additivity.

\begin{lem} \label{lem.quad.add} {\bf (Quadratic Variation Additivity)} \\
Consider a vector $x \in \h$ and a Gaussian random variable $\xi \dist \pi_0$ 
and a real number $\alpha \in \RR$.
Suppose that the vector $x \in \h$ possesses a finite quadratic variation
$V(x)<+\infty$. 
Then almost surely the vector $x+\alpha \xi \in \h$ possesses a quadratic
variation 
that is equal to
\begin{equs}
V(x + \alpha \xi) \; = \; V(x) + \alpha^2. 
\end{equs} 
\end{lem}

\begin{proof}
Let us define $V_N \eqdef N^{-1} \sum_{1}^N \frac{\bra{x,\phi_j} \cdot
\bra{\xi,\phi_j}}{\lambda_j^2}$. To prove Lemma \ref{lem.quad.add} 
it suffices to prove that almost surely the following limit holds
\begin{equs}
\lim_{N \to \infty} \, V_N \:=\:  0.
\end{equs}
Borel-Cantelli Lemma shows that it suffices to prove that for every fixed $\eps
> 0$ 
we have $\sum_{N \geq 1} \; \PP\big[ \big|V_N \big| > \eps \big] < \infty$.
Notice then that $V_N$ is a centred Gaussian random variables with variance 
\begin{equs}
\textrm{Var}(V_N) = \frac{1}{N} \, \Big( N^{-1} \sum_1^N
\frac{\bra{x,\phi_j}^2}{\lambda_j^2}\Big) \asymp \frac{V(x)}{N}.
\end{equs}
It readily follows that $\sum_{N \geq 1} \; \PP\big[ \big|V_N \big| > \eps \big]
< \infty$, 
finishing the proof of the Lemma.
\end{proof}

\subsection{Large $k$ Behaviour of Quadratic Variation for pCN}

The pCN algorithm at temperature $\tau>0$ 
and discretization parameter 
$\delta>0$ proposes a move from $x$ to $y$
according to the dynamics
\begin{equs}
y = (1-2 \delta)^{\frac12} x + (2\delta \tau)^{\frac12} \, \xi
\qquad \text{with} \qquad
\xi \dist \pi_0.
\end{equs}
This move is accepted with probability $\alpha^{\delta}(x,y)$. In this case,
Lemma 
\ref{lem.quad.add} shows that if the quadratic variation
$V(x)$ exists then the quadratic variation 
of the proposed move $y \in \h$ exists and satisfies 
\begin{equs} \label{e.delta.quad.when.accepted}
\frac{V(y)-V(x)}{\delta} \;=\; -2 (V(x) - \tau).
\end{equs}
Consequently, one can prove that for any finite time step $\delta > 0$ and 
temperature $\tau>0$ the quadratic variation of the MCMC algorithm 
converges to $\tau$.

\begin{prop} {\bf (Limiting Quadratic Variation)} \label{prop.limiting.qv}
Let Assumptions \ref{ass:1} hold and $\{x^{k,\delta}\}_{k \geq 0}$ be the Markov
chain 
of section \ref{sec:algorithm}. Then almost surely the quadratic variation 
of the Markov chain converges to $\tau$,
\begin{equs}
\lim_{k \to \infty} \, V(x^{k,\delta}) \;=\; \tau.
\end{equs}
\end{prop}
\begin{proof}
Let us first show that the number of accepted moves is infinite.
If this were not the case, the Markov chain would eventually reach a position 
$x^{k,\delta} = x \in \h$ such that all subsequent proposals 
$y^{k+l} = (1-2\delta)^{\frac12} \, x^k + (2 \tau \delta)^{\frac12} \,
\xi^{k+l}$ would be refused. 
This means that the $\iid$ Bernoulli random variables 
$\gamma^{k+l} = \textrm{Bernoulli} \big(\alpha^{\delta}(x^k,y^{k+l}) \big)$
satisfy 
$\gamma^{k+l} = 0$ for all $l \geq 0$. This can only happen with probability
zero. 
Indeed, since  $\PP[\gamma^{k+l} = 1] > 0$, one can use Borel-Cantelli Lemma 
to show that almost surely there exists $l \geq 0$ such that $\gamma^{k+l} = 1$.
To conclude the proof of the Proposition, notice then that
the sequence $\{ u_{k} \}_{k \geq 0}$ defined by 
$u_{k+1}-u_k \;=\; -2 \delta (u_k - \tau)$ converges to $\tau$.
\end{proof}

\subsection{Fluid Limit for Quadratic Variation of pCN}

To gain further insight into the rate at which the
limiting behaviour of the quadratic variation is
observed for pCN we derive an ODE ``fluid limit''
for the Metropolis-Hastings algorithm.
We introduce the continuous time process 
$t \mapsto v^{\delta}(t)$ 
defined as continuous piecewise linear
interpolation of the process $k \mapsto V(x^{k,\delta})$,
\begin{equs} \label{e.speed.up.qv}
v^{\delta}(t)= \frac{1}{\delta} \, (t-t_k) \, V(\, x^{k+1,\delta} \,)
+ \frac{1}{\delta} \, (t_{k+1}-t) \, V( \, x^{k,\delta} \,)
\quad \text{for} \quad 
t_k \leq t < t_{k+1}.
\end{equs}
Since the acceptance probability of pCN approaches
one as $\delta \to 0$ (see Corollary \ref{cor.alpha.bar})
Equation \eqref{e.delta.quad.when.accepted} shows 
heuristically that the trajectories 
of the process $t \mapsto v^{\delta}(t)$ 
should be well approximated 
by the solution of the (non-stochastic) differential equation 
\begin{equs} \label{eqn:limit:ODE}
    \dot{v}   &= -2 \, (v - \tau).
\end{equs}
We prove such a result, in the sense of convergence
in probability in $C([0,T],{\mathbb R})$:

\begin{thm} \label{th.fluid.limit} {\bf (Fluid Limit For
Quadratic Variation)}
Let Assumptions \ref{ass:1} hold.
Let the Markov chain $x^{\delta}$ start at fixed position $x_* \in \h^s$. 
Assume that $x_* \in \h$ possesses a finite quadratic variation, $V(x_*) <
\infty$.
Then the function $v^{\delta}(t)$
converges in probability in $C([0,T], \RR)$, as $\delta$ goes to zero,  to the
solution of the differential equation
\eqref{eqn:limit:ODE} with initial condition $v_0 = V(x_*)$.
\end{thm}

\noindent
As already indicated, the heart of the proof consists in 
showing that the acceptance probability 
of the algorithm converges to one as $\delta$ goes to zero. 
We prove such a result as Lemma \ref{lem.accept.moves} 
below, and then proceed to prove Theorem \ref{th.fluid.limit}.
To this end we introduce $t^{\delta}(k)$,  
the number of accepted moves,
\begin{equs}
t^{\delta}(k) \; \eqdef \; \sum_{l \leq k} \gamma^{l, \delta},
\end{equs}
where $\gamma^{l, \delta} = \textrm{Bernoulli}(\alpha^{\delta}(x,y))$ 
is the Bernoulli random variable defined in 
Equation \eqref{eqn:mark:dynamic}. Since the acceptance probability of the
algorithm 
converges to $1$ as $\delta \to 0$, the approximation $t^{\delta}(k) \approx k$
holds. 
In order to prove 
a fluid limit result on the interval $[0,T]$ one needs to prove that the
quantity 
$\big| t^{\delta}(k) - k \big|$ is small when compared to $\delta^{-1}$. The
next Lemma 
shows that such a bounds holds uniformly on the interval $[0,T]$.

\begin{lem} {\bf (Number of Accepted Moves)} \label{lem.accept.moves}
Let Assumptions \ref{ass:1} hold. The number of accepted moves
$t^{\delta}(\cdot)$ 
verifies
\begin{equs}
\lim_{\delta \to 0} \quad \sup \big\{ \delta \cdot \big| t^{\delta}(k) -k \big| 
\; : 0 \leq k \leq T \delta^{-1} \big\} = 0
\end{equs}
where the convergence holds in probability.
\end{lem}

\begin{proof}
The proof is given in Appendix \ref{app:accmoves}.
\end{proof}

\noindent
We now complete the proof of Theorem \ref{th.fluid.limit} using
the key Lemma \ref{lem.accept.moves}.

\begin{proof}[of Theorem \ref{th.fluid.limit}]
The proof consists in showing that the trajectory of the 
quadratic variation process behaves as if all the move were 
accepted. The main ingredient is the uniform lower bound 
on the acceptance probability given by Lemma \ref{lem.accept.moves}.

\noindent
Recall that $v^{\delta}(k \delta) = V(x^{k,\delta})$.
Consider the piecewise linear
function $\what{v}^{\delta}( \cdot) \in C([0,T], \RR)$ defined 
by linear interpolation of the values
$\hat{v}^{\delta}(k \delta) = u^{\delta}(k)$ 
and where the sequence $\{u^{\delta}(k) \}_{k \geq 0}$ satisfies $u^{\delta}(0)
= V(x_*)$ and 
\begin{equs}
u^{\delta}(k+1) - u^{\delta}(k) \;=\; -2 \delta \, (u^{\delta}(k) - \tau).
\end{equs}
The value $u^{\delta}(k) \in \RR$ represents the quadratic variation of
$x^{k,\delta}$ 
if the $k$ first moves of the MCMC algorithm had been accepted.
One can readily check that as $\delta$ goes to zero 
the sequence of continuous 
functions $\hat{v}^{\delta}(\cdot)$ converges in $C([0,T], \RR)$ to 
the solution $v(\cdot)$ of the differential 
equation \eqref{eqn:limit:ODE}. Consequently, to prove Theorem
\ref{th.fluid.limit} 
it suffices to show that for any $\eps > 0$ we have
\begin{equs} \label{e.sup.conv.prob}
\lim_{\delta \to 0} \; \PP \Big[ \sup \Big\{  \big| V(x^{k,\delta}) -
u^{\delta}(k) \big| \, : \, k \leq \delta^{-1} T \Big\} \, > \, \eps \Big] \;=\;
0.
\end{equs}
The definition of the number of accepted moves $t^{\delta}(k)$ is such that 
$V(x^{k,\delta})  = u^{\delta}(t^{\delta}(k))$. Note that
\begin{equs}
\label{eq:ubd}
u^{\delta}(k)=(1-2\delta)^k u_0+\bigl(1-(1-2\delta)^k\bigr)\tau.
\end{equs}
Hence, for any integers $t_1, t_2 \geq 0,$ we have 
$\big|u^{\delta}(t_2) - u^{\delta}(t_1) \big| \leq  \big| u^{\delta}(|t_2-t_1|)
- u^{\delta}(0) \big|$
so that 
\begin{equs}
\big|V(x^{k,\delta})  - u^{\delta}(k) \big| 
\;=\; \big| u^{\delta}(t^{\delta}(k))  - u^{\delta}(k) \big| 
\;\leq\; \big| u^{\delta}(k - t^{\delta}(k))  - u^{\delta}(0) \big|. 
\end{equs}
Equation \eqref{eq:ubd} shows that $|u^{\delta}(k)-u^{\delta}(0)| \; \lesssim \;
\big( 1 - (1-2\delta)^k \big)$.
This implies that
\begin{equs}
\big| V(x^{k,\delta}) - u^{\delta}(k) \big|
\;\lesssim\; 1 - (1-2\delta)^{k-t^{\delta}(k)} 
\;\lesssim\; 1 - (1-2\delta)^{\delta^{-1} \, S} 
\end{equs}
where $S = \sup \big\{ \delta \cdot \big| t^{\delta}(k) -k \big|  \; : 0 \leq k
\leq T \delta^{-1} \big\}$.
Since for any $a>0$ we have $1-(1-2\delta)^{a \delta^{-1}} \to 1-e^{-2a}$, 
Equation  \eqref{e.sup.conv.prob} follows 
if one can prove that as $\delta$ goes to zero the supremum $S$ 
converges to zero in probability: this is precisely the content of Lemma 
\ref{lem.accept.moves}. 
This concludes the proof of Theorem \ref {th.fluid.limit}.
\end{proof}

\section{Numerical Results}
\label{sec:6}
In this section, we present some numerical
simulations demonstrating our results.
We consider the minimization of a functional 
$J(\cdot)$ defined on the Sobolev space 
$H^1_0(\RR).$ Note that functions $x \in H^1_0([0,1])$ are continuous 
and satisfy $x(0) = x(1) = 0;$ thus $H^1_0(\RR) \subset C^0([0,1]) \subset L^2(0,1)$.
For a given real parameter $\lambda > 0$, the functional 
$J:H^1_0([0,1]) \to \RR$ is composed of two competitive terms,
as follows:
\begin{equs}
J(x) \; = \; \frac{1}{2} \int_0^1 \bigl|\dot{x}(s) \bigr|^2 \, ds \;+\; 
\frac{\lambda}{4} \, \int_0^1 \bigl(x(s)^2 - 1\bigr)^2 \, ds.
\label{eq:Jex}
\end{equs}
The first term penalizes functions that deviate from being flat,
whilst the second term 
penalizes functions that deviate from one in absolute value. 
Critical points of the functional $J(\cdot)$ solve the 
following Euler-Lagrange equation:
\begin{equs} \label{e.euler.lagrange}
\ddot{x} + \lambda \, x(1-x^2) &=0\\
x(0)=x(1)&=0. 
\end{equs}
Clearly $x \equiv 0$ is a solution 
for all $\lambda \in {\mathbb R}^+.$
If $\lambda \in (0,\pi^2)$ then this is the unique
solution of the Euler-Lagrange 
equation and is the global minimizer of $J$. For each integer
$k$ there is a supercritical bifurcation at parameter
value $\lambda=k^2\pi^2.$ For $\lambda>\pi^2$ there
are two minimizers, both of one sign and one being minus the other.
The three different solutions of \eqref{e.euler.lagrange} 
which exist for $\lambda=2\pi^2$
are displayed in Figure \ref{fig.sol.euler.lagrange}, 
at which value the zero (blue dotted)
solution is a saddle point, and the two green solutions are
the global minimizers of $J$. These properties of $J$ are
overviewed in, for example, \cite{henry}. 
We will show how these global minimizers can emerge from
an algorithm whose only ingredients are an ability
to evaluate
$\Psi$ and to sample from the Gaussian measure
with Cameron-Martin norm $\int_0^1 |{\dot x}(s)|^2 ds.$
We emphasize that we are not advocating this as the optimal
method for solving the Euler-Lagrange 
equations \eqref{e.euler.lagrange}. We have chosen this
example for its simplicity, in order to illustrate
the key ingredients of the theory developed in this paper.

\begin{figure}[hh]
  \centering
  \includegraphics[width= 4in]{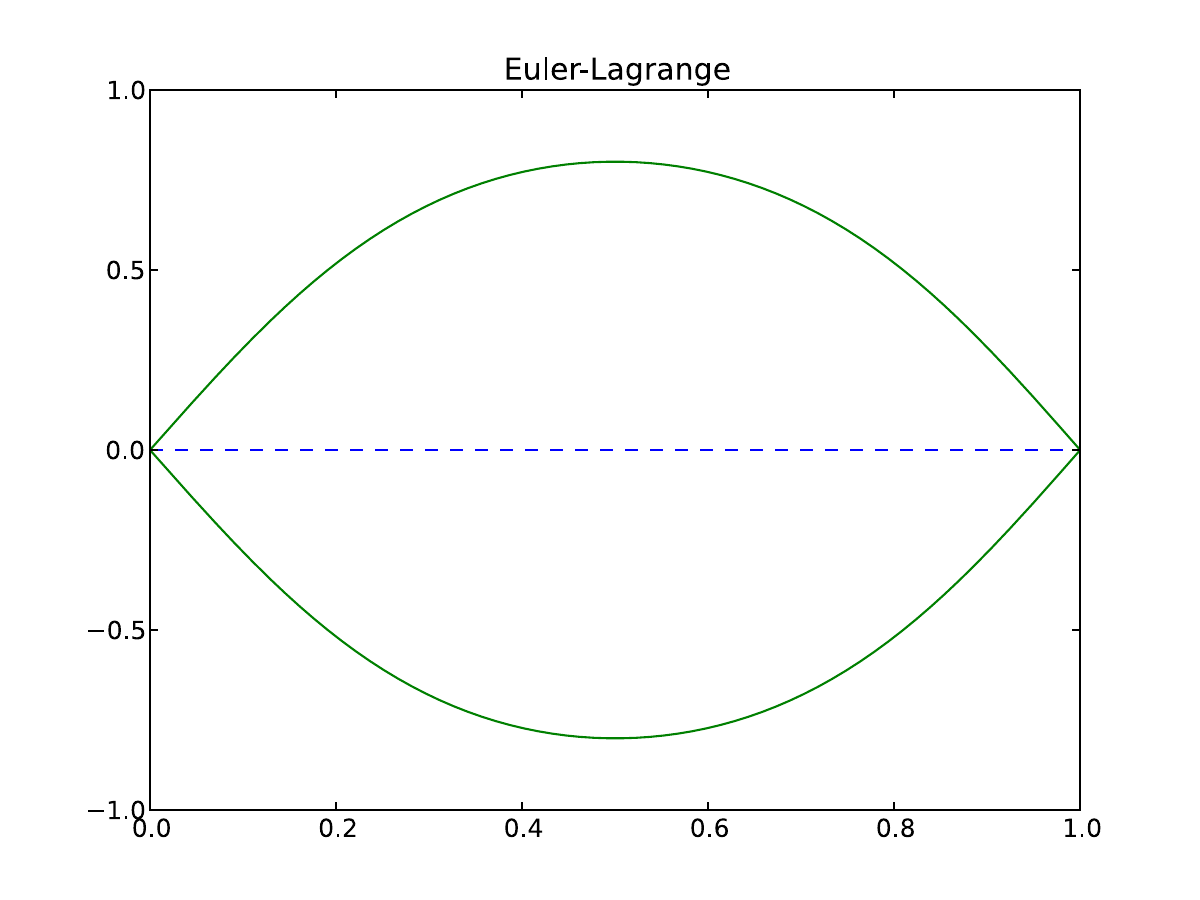}
  \caption{ \label{fig.sol.euler.lagrange} The three solutions of the
Euler-Lagrange Equation 
  \eqref{e.euler.lagrange}  for $\lambda = 2 \pi^2$. Only the two non-zero
solutions 
  are global minimum of the functional $J(\cdot)$. The dotted solution is a
local maximum 
  of $J(\cdot)$.} 
\end{figure}

The pCN algorithm to minimize $J$ given by
\eqref{eq:Jex} is implemented on $L^2([0,1])$. 
Recall from \cite{Dapr:Zaby:92} 
that the Gaussian measure $\Normal(0,C)$ 
may be identified by finding the covariance
operator for which the $H^1_0([0,1])$ norm 
$\|x\|^2_C \eqdef \int_0^1 \bigl|\dot{x}(s)\bigr|^2 \, ds$
is the Cameron-Martin norm.
In \cite{Hair:etal:05} it is shown that the
Wiener bridge measure $\WW_{0 \to 0}$ on $L^2([0,1])$
has precisely this Cameron-Martin norm; indeed it is
demonstrated that $C^{-1}$ is the densely defined
operator $-\frac{d^2}{ds^2}$ with $D(C^{-1})=H^2([0,1])\cap H^1_0([0,1]).$ 
In this regard it is also instructive to adopt the
physicists viewpoint that
$$\WW_{0 \to 0}(dx) \propto \exp \Big( -\frac{1}{2}  \int_0^1 
\bigl|\dot{x}(s)\bigr|^2 \,ds\Big) \, dx$$ 
although, of course,
there is no Lebesgue measure in infinite dimensions. 
Using an integration by parts, together with the boundary
conditions on $H^1_0([0,1])$, then gives
$$\WW_{0 \to 0}(dx) \propto \exp \Big( \frac{1}{2}  \int_0^1 
x(s)\frac{d^2 x}{ds^2}(s)\,ds \Big) \, dx$$ 
and the inverse of $C$ is clearly identified as 
the differential operator above.
See \cite{chorinhald} for basic discussion of the
physicists viewpoint on Wiener measure. 
For a given temperature parameter $\tau$ the Wiener bridge
measure $\WW_{0 \to 0}^{\tau}$ on $L^2([0,1])$ is defined as 
the law of
$\big\{ \sqrt{\tau} \, W(t) \big\}_{t \in [0,1]}$ 
where $\{ W(t) \}_{t \in [0,1]}$ is a standard 
Brownian bridge on $[0,1]$ drawn from $\WW_{0 \to 0}$.

The posterior distribution $\pi^{\tau}(dx)$ is defined 
by the change of probability formula
\begin{equs}
\frac{d \pi^{\tau}}{d \WW_{0 \to 0}^{\tau}}(x) \;\propto\; e^{-\Psi(x)} 
\qquad \text{with} \qquad
\Psi(x) = \frac{\lambda}{4} \, \int_0^1 \bigl(x(s)^2-1\bigr)^2 \, ds.
\end{equs}
Notice that 
$\pi^{\tau}_0\bigl(H^1_0([0,1)\bigr)  
=\pi^{\tau}\bigl(H^1_0([0,1)\bigr) = 0$ since a Brownian bridge is almost
surely 
not differentiable anywhere on $[0,1]$. For this
reason, the algorithm is 
implemented on $L^2([0,1])$ even though the functional 
$J(\cdot)$ is defined on the Sobolev space $H^1_0([0,1])$. 
In terms of Assumptions \ref{ass:1}(1) we have $\kappa=1$
and the measure $\pi^{\tau}_0$ is supported on 
${\cal H}^r$ if and only if $r<\frac12$, 
see Remark \ref{rem:one}; 
note also that $H^1_0([0,1])={\cal H}^1.$ 
Assumption \ref{ass:1}(2) is satisfied for any choice
$s \in [\frac14,\frac12)$ because ${\cal H}^s$ is embedded
into $L^4([0,1])$ for $s \ge \frac14.$ We add here that 
Assumptions \ref{ass:1}(3-4) do not hold globally, but
only locally on bounded sets, but the numerical results
below will indicate that the theory developed in this paper
is still relevant and could be extended to nonlocal
versions of Assumptions \ref{ass:1}(3-4), with considerable
further work.

Following 
section \ref{sec:algorithm}, the pCN Markov chain at temperature $\tau>0$ 
and time discretization $\delta>0$ proposes moves 
from $x$ to $y$ according to
\begin{equs}
y = (1-2\delta)^{\frac12} \, x \;+\; (2\delta \tau)^{\frac{1}{2}} \, \xi
\end{equs}
where $\xi \in C([0,1], \RR)$ is a standard Brownian bridge on $[0,1]$. 
The move $x \to y$ is accepted with probability 
$\alpha^{\delta}(x,\xi) = 1 \wedge \exp\big( -\tau^{-1}[\Psi(y) -
\Psi(x)]\big)$. 
Figure \ref{fig.error} displays the convergence of the 
Markov chain $\{x^{k,\delta}\}_{k \geq 0}$ to a minimizer of the functional 
$J(\cdot)$. 
Note that this convergence is not shown with respect to
the space $H^1_0([0,1])$ on which $J$ is defined, but
rather in $L^2([0,1])$; indeed
$J(\cdot)$ is almost surely infinite when evaluated at
samples of the pCN algorithm, precisely because 
$\pi^{\tau}_0\bigl(H^1_0([0,1)\bigr) = 0,$ as discussed above.

\begin{figure}[hh]
\centering
\includegraphics[width= 4in]{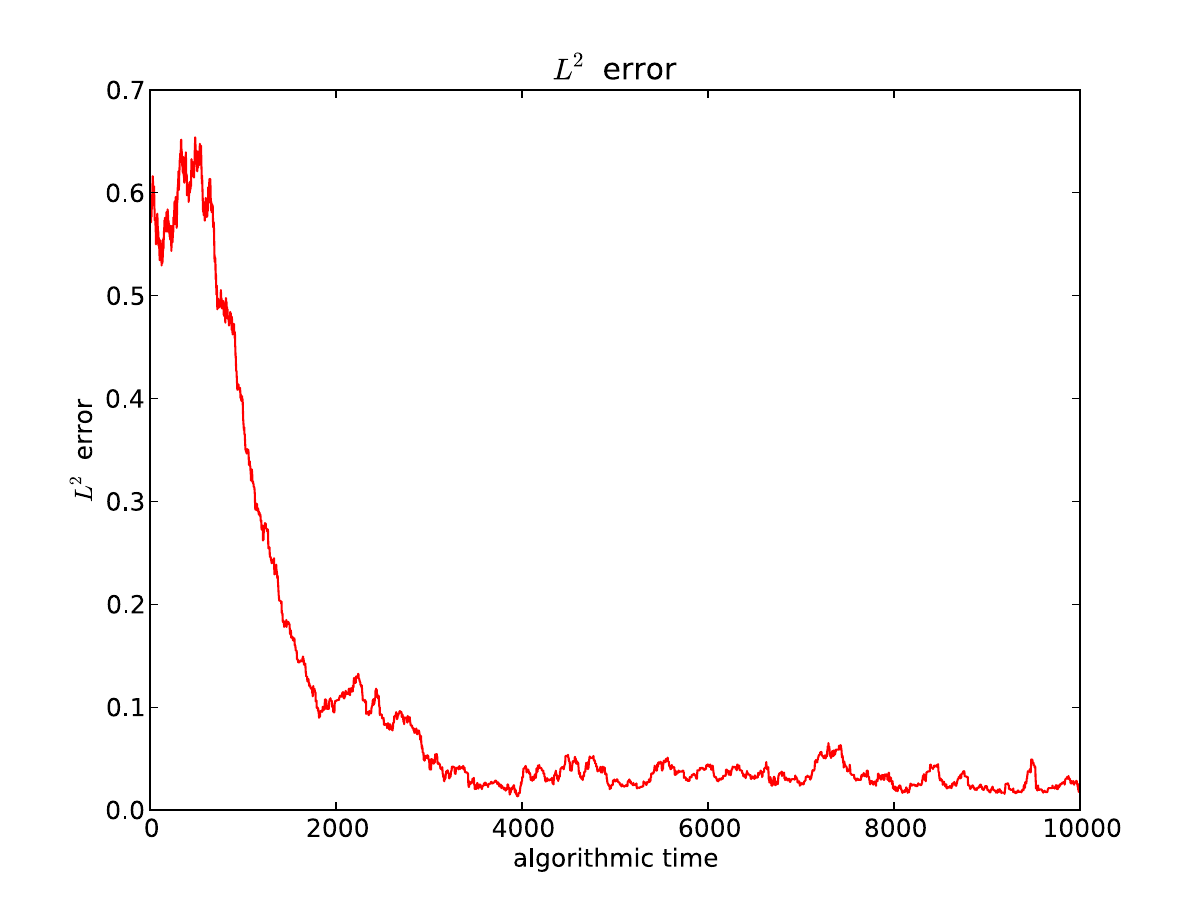}
\caption{ \label{fig.error} pCN parameters: $\lambda = 2\pi^2$, $\delta =
1.10^{-2}$, $\tau = 1.10^{-2}$. The algorithm is started at the zero function,
$x^{0,\delta}(t) = 0$ for $t \in [0,1]$.
After a transient phase, the algorithm fluctuates around a global minimizer of 
functional $J(\cdot)$. The $L^2$ error $\|x^{k,\delta} -
(\text{minimizer})\|_{L^2}$ is plotted as a function of the algorithmic time
$k$.}
 \end{figure}

Of course the algorithm does not converge {\em exactly} 
to a minimizer 
of $J(\cdot),$ but fluctuates in a neighborhood of it. As described in the
introduction 
of this article, in a finite dimensional setting 
the target probability distribution $\pi^{\tau}$ has Lebesgue 
density proportional to $\exp\big( -\tau^{-1} \, J(x)\big)$. This intuitively
shows that
the size of the fluctuations around the minimum of the functional $J(\cdot)$ 
are of size proportional to $\sqrt{\tau}$. 
Figure \ref{fig.error_tau} shows this phenomenon 
on log-log scales: the asymptotic mean 
error $\EE\big[ \|x - (\text{minimizer})\|_2\big]$ 
is displayed as a function of the temperature $\tau$. 
Figure \ref{fig.qv} illustrates Theorem \ref{th.fluid.limit}. One can 
observe the path $\{ v^{\delta}(t) \}_{t \in [0,T]}$ for a finite time 
step discretization parameter $\delta$ as well as the limiting path 
$\{v(t)\}_{t \in [0,T]}$ that is solution of the 
differential equation \eqref{eqn:limit:ODE}.

\begin{figure}[hh]
\centering
\includegraphics[width= 4in]{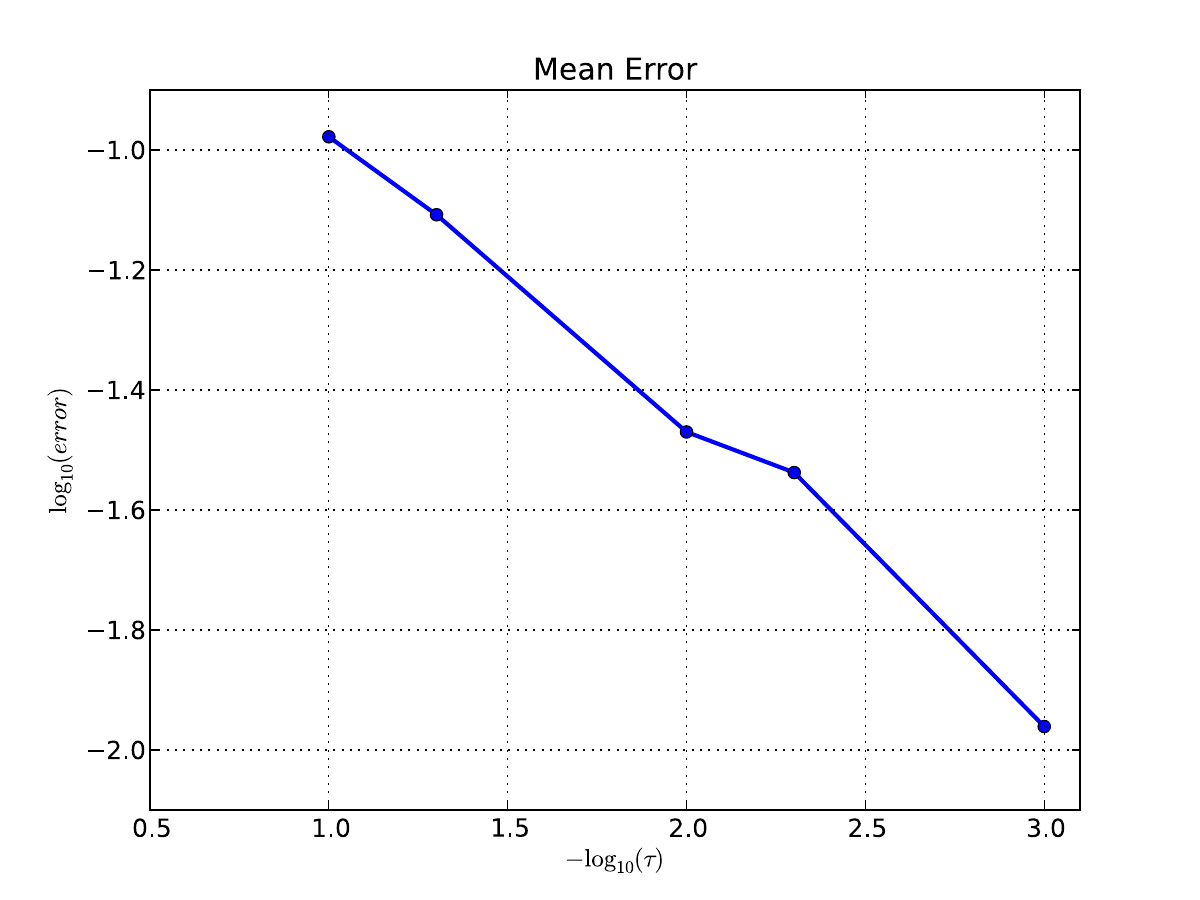}
\caption{ \label{fig.error_tau} }
Mean error $\EE\big[ \|x - (\text{minimizer})\|_2\big]$ as a function of the
temperature $\tau$.
 \end{figure}

\begin{figure}[htb]
\centering
\includegraphics[width= 4in]{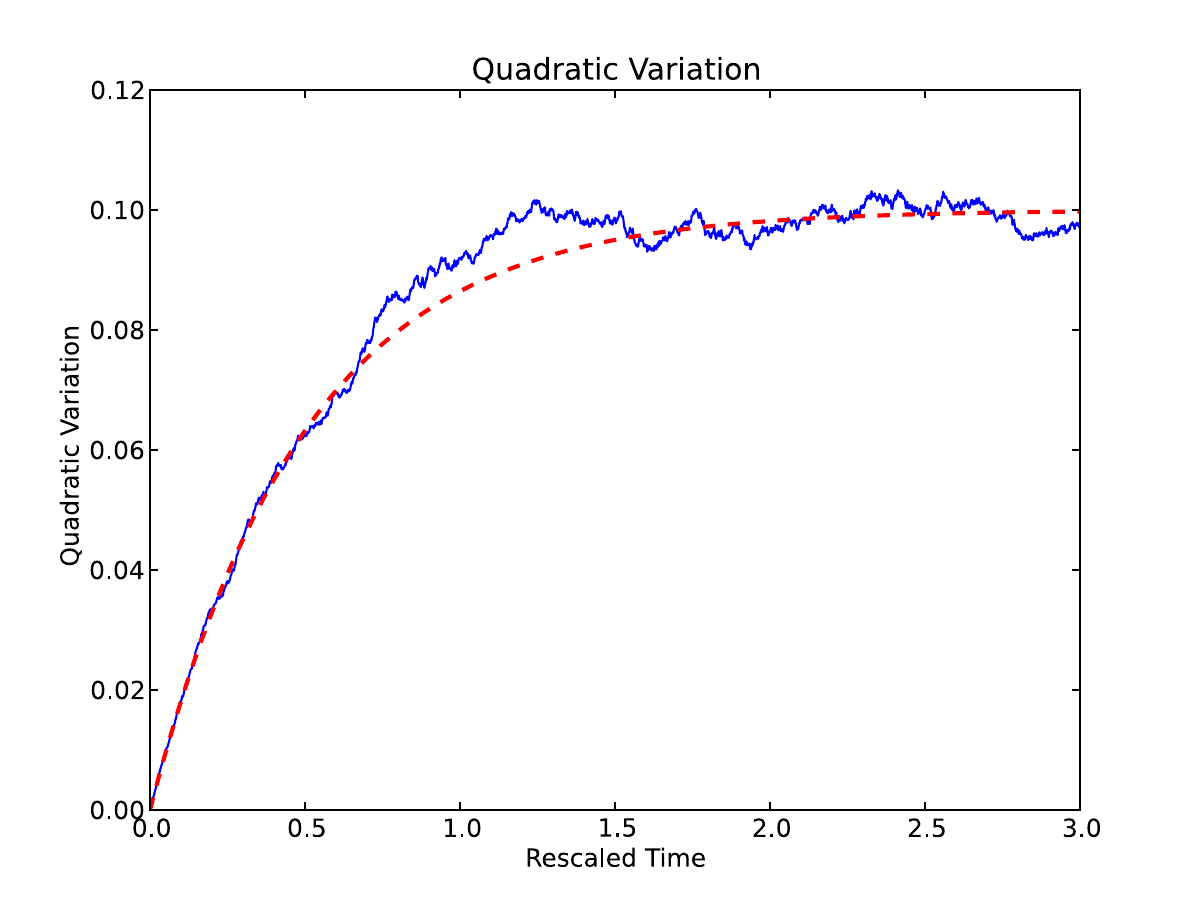}
\caption{ \label{fig.qv} 
pCN parameters: $\lambda=2 \,\pi^2$, $\tau = 1.10^{-1}$, $\delta = 1.10^{-3}$
and 
the algorithm starts at $x^{k,\delta}=0$. The rescaled quadratic variation
process (full 
line) behaves as the solution of the differential equation (dotted line), as
predicted 
by Theorem \ref{th.fluid.limit}. The quadratic variation converges to $\tau$, as
described 
by Proposition \ref{prop.limiting.qv}.
}
 \end{figure}
 
\section{Conclusion}
\label{sec:7}

There are two useful perspectives on the material 
contained in this paper, one concerning optimization
and one concerning statistics. We now detail these
perspectives.

\begin{itemize}

\item {\bf Optimization} We have demonstrated a class of
algorithms to minimize the functional $J$ given by
\eqref{eq:I}. The Assumptions \ref{ass:1} encode the
intuition that the quadratic part of $J$ dominates. Under
these assumptions we study the properties of an
algorithm which requires only the evaluation of $\Psi$
and the ability to draw samples from Gaussian measures
with Cameron-Martin norm given by the quadratic part of $J$.
We demonstrate that, in a certain parameter limit,
the algorithm behaves like a noisy gradient flow for
the functional $J$ and that, furthermore, the size
of the noise can be controlled systematically.
The advantage of constructing algorithms
on Hilbert space is that they are robust to finite
dimensional approximation. We turn to this point in the
next bullet. 

\item {\bf Statistics} The algorithm that we use is
a Metropolis-Hastings method
with an Onrstein-Uhlenbeck proposal which we refer
to here as pCN, as in \cite{CRSW13}. The proposal takes the form for
$\xi \sim \Normal(0,C)$,
$$y=\bigl(1-2\delta\bigr)^{\frac12}x+\sqrt{2\delta\tau}\xi$$
given in \eqref{eq:prop}.  The proposal is constructed
in such a way that the algorithm is defined on infinite
dimensional Hilbert space and may be viewed as
a natural analogue of a random walk Metropolis-Hastings
method for measures defined via density with respect to
a Gaussian. It is instructive to contrast this with
the standard random walk method S-RWM with proposal 
$$y=x+\sqrt{2\delta\tau}\xi.$$
Although the proposal for S-RWM differs only through
a multiplicative factor in the systematic component,
and thus implementation of either is practically
identical, the S-RWM method is not defined on
infinite dimensional Hilbert space. 
This turns out to matter if we compare both
methods when applied in ${\mathbb R}^N$ for $N \gg 1$, as would
occur if approximating a problem in infinite
dimensional Hilbert space: in this setting
the S-RWM method requires the choice $\delta={\cal O}(N^{-1})$
to see the diffusion (SDE) limit \cite{mattingly2011spde}
and so requires ${\cal O}(N)$ steps to see ${\cal O}(1)$ 
decrease in the objective function, or to draw independent
samples from the target measure; in contrast the pCN 
produces a diffusion limit for $\delta \to 0$ independently
of $N$ and so requires ${\cal O}(1)$ steps to see
${\cal O}(1)$ 
decrease in the objective function, or to draw independent
samples from the target measure. Mathematically this 
last point is manifest in the fact that we may 
take the limit $N \to \infty$ (and
thus work on the infinite dimensional Hilbert space) followed
by the limit $\delta \to 0.$

\end{itemize}

The methods that we employ for the derivation of the
diffusion (SDE) limit use a combination of ideas from
numerical analysis and the weak convergence
of probability measures. This approach is encapsulated
in Lemma \ref{lem:diff_approx}
which is structured in such a way
that it, or variants of it,
may be used to prove diffusion limits for
a variety of problems other than the one considered here.

\vspace{0.1in}

\appendix

\section{Proofs of Lemmas; Section 4} \label{app:sec4calc}
\begin{proof}[Proof of Lemma \ref{lem.acc.prob.estim}]
Let us introduce the two $1$-Lipschitz functions $h,h_*:\RR \to \RR$ defined by
\begin{equs} \label{e.lip.h}
h(x) = 1 \wedge e^x \qquad \text{and} \qquad h_*(x) = 1 + x \, 1_{\{x < 0\}}.
\end{equs}
The function $h_*$ is a first order approximation of $h$ in a neighborhood of
zero and we have
\begin{equs}
\alpha^{\delta}(x,\xi) \;=\; h\Big( -\frac{1}{\tau}\{\Psi(y)-\Psi(x)\} \Big) 
\qquad \text{and} \qquad
\bar{\alpha}^{\delta}(x,\xi) \;=\; h_* \Big( -\sqrt{\frac{2\delta}{\tau}} \;
\bra{\nabla \Psi(x), \xi} \Big)
\end{equs}
where the proposal $y$ is a function of $x$ and $\xi$, as described in Equation
\eqref{eqn:RWMprop}.
Since $h_*(\cdot)$ is close to $h(\cdot)$ in a neighborhood of zero, 
the proof is finished once it is proved that
$-\frac{1}{\tau}\{\Psi(y)-\Psi(x)\}$ is close to 
$-\sqrt{\frac{2\delta}{\tau}} \; \bra{\nabla \Psi(x), \xi}$.
We have $\EE_x\Big[ |\alpha^{\delta}(x,\xi) - \bar{\alpha}^{\delta}(x,\xi)|^p
\Big] \;\lesssim\; A_1 + A_2$
where the quantities $A_1$ and $A_2$ are given by
\begin{equs}
A_1 &= \EE_x\Big[ \big| h \big(-\frac{1}{\tau}\{\Psi(y)-\Psi(x)\} \big) 
\;-\; h \big(-\sqrt{\frac{2\delta}{\tau}} \; \bra{\nabla \Psi(x), \xi} \big)
\big|^p \Big] \\
A_2 &= \EE_x\Big[ \big| h\big(-\sqrt{\frac{2\delta}{\tau}} \; \bra{\nabla
\Psi(x), \xi} \big) 
\;-\; h_* \big(- \sqrt{\frac{2\delta}{\tau}} \; \bra{\nabla \Psi(x), \xi} \big)
\big|^p \Big] .
\end{equs}
By Lemma \ref{lem:lipshitz+taylor}, the first order Taylor approximation of
$\Psi$ is 
controlled, 
$\big| \Psi(y)-\Psi(x) - \bra{\nabla \Psi(x),y-x} \big| \lesssim \|y-x\|_s^2$. 
The definition of the proposal $y$ given in Equation \eqref{eqn:RWMprop}  shows
that 
$\|(y-x)-\sqrt{2\delta \tau} \xi\|_s \lesssim \delta \|x\|_s$.
Assumptions \ref{ass:1} state that for $z \in \h^s$ we have $\bra{\nabla
\Psi(x),z} \lesssim \big(1+\|x\|_s \big) \cdot \|z\|_s$. 
Since the function $h(\cdot)$ is $1$-Lipschitz it follows that
\begin{equs} \label{e.estim.A}
A_1 
&= \EE_x\Big[ \big|h\big(-\frac{1}{\tau}\{\Psi(y)-\Psi(x)\} \big) -
h\big(-\sqrt{\frac{2\delta}{\tau}} \; \bra{\nabla \Psi(x), \xi} \big) \big|^p
\Big] \\
&\lesssim \EE_x\Big[ \big| \Psi(y)-\Psi(x)-\bra{\nabla \Psi(x),y-x} \big|^p 
\;+\; \big| \bra{\nabla \Psi(x),y-x- \sqrt{2\delta \tau} \xi} \big|^p \Big] \\
&\lesssim \EE_x\Big[ \|y-x\|_s^{2p} 
\;+\; (1+\|x\|_s^p) \cdot (\delta \; \|x\|_s)^p \Big] 
\;\lesssim\; \delta^{p} \, (1+\|x\|_s^{2p}).
\end{equs}
Lemma \ref{lem:size:prop} has been used to control the size of $\EE_x \big[
\|y-x\|^p \big]$.
To bound $A_2$, notice that for $z \in \RR$ we have $|h(z)-h_*(z)| \leq \frac12
\, z^2$. 
Therefore the quantity $A_2$ can be bounded by
\begin{equs} \label{e.estim.B}
A_2 
&\lesssim \EE_x\Big[ |\sqrt{\delta} \; \bra{\nabla \Psi(x), \xi}|^{2p}\Big]
\;\lesssim\; \delta^{p} \; \EE_x\Big[ (1+\|x\|_s^{2p}) \, \|\xi\|_s^{2p} \Big] 
\;\lesssim\; \delta^{p} \; (1+\|x\|_s^{2p}).
\end{equs}
Estimates \eqref{e.estim.A} and \eqref{e.estim.B} together 
give Equation \eqref{e.acc.asymp}.
\end{proof}

\begin{proof}[Proof of Corollary \ref{cor.alpha.bar}]
Let us prove Equations \eqref{e.bound.alpha.apha.bar.1} and
\eqref{e.bound.alpha.apha.bar.2}.
\begin{itemize}
\item 
Lemma \ref{lem.acc.prob.estim} and Jensen's inequality give 
Equation \eqref{e.bound.alpha.apha.bar.1}.

\item 
To prove \eqref{e.bound.alpha.apha.bar.2}, one can suppose $\delta^{\frac{p}{2}}
\|x\|^p_s \leq 1$. Indeed,
if $\delta^{\frac{p}{2}} \|x\|^p_s \geq 1$, we have
\begin{equs}
\EE_x \Big[ \big|\alpha^{\delta}(x,\xi)  \;-\;  1 \big|^p \Big] \;\lesssim\; 1 
\;\leq\; \delta^{\frac{p}{2}} \|x\|^p_s 
\;\leq\; \delta^{\frac{p}{2}} \; (1 + \|x\|^p_s),
\end{equs}
which gives the result.
We thus suppose from now on that $\delta^{\frac{p}{2}} \|x\|_s \leq 1$.
Under Assumptions \ref{ass:1} we have $\|\nabla \Psi(x)\|_{-s} \lesssim
1+\|x\|_s$.
Lemma \ref{lem:lipshitz+taylor} shows that for all $x,y \in \h^s$ we have
$\big|\Psi(y)-\Psi(x) - \bra{\nabla \Psi(x), y-x } \big| \lesssim \|y-x\|_s^2$. 
The function $h(x) = 1 \wedge e^x$ is $1$-Lipschitz, 
$\alpha^{\delta}(x,\xi) = h\big(-\frac{1}{\tau}[\Psi(y)-\Psi(x)] \big)$ and
$h(0)=1$.
Consequently,
\begin{equs}
\EE_x \Big[ \big|\alpha^{\delta}(x,\xi)  \;-\;  1 \big|^p \Big]
&=\EE_x \Big[ \big| h\big(-\frac{1}{\tau}[\Psi(y)-\Psi(x)] \big) \;-\; h(0)
\big|^p \Big]\\
&\lesssim \EE_x\big[|\Psi(y)-\Psi(x)|^p \big]
\;\lesssim\; \EE_x\big[ |\bra{\nabla \Psi(x), y-x }|^p + \|y-x\|_s^{2p} \big]\\
&\lesssim (1+\|x\|_s^p) \cdot \EE_x\big[ \|y-x\|^p_s] + \EE_x
\big[\|y-x\|_s^{2p} \big].
\end{equs}
By Lemma \ref{lem:size:prop}, for any integer $\beta \geq 1$ we have 
$\EE_x\big[ \|y-x\|^{\beta}_s \big] \lesssim \delta^{\beta} \|x\|_s^{\beta} +
\delta^{\frac{\beta}{2}}$
so that the assumption $\delta^{\frac{p}{2}} \|x\|^p_s \leq 1$ leads to
\begin{equs}
\EE_x \Big[ \big|\alpha^{\delta}(x)  \;-\;  1 \big|^p \Big]
&\lesssim (1+\|x\|_s^p) \cdot (\delta^{p} \|x\|_s^{p} + \delta^{\frac{p}{2}}) +
(\delta^{2p} \|x\|_s^{2p} + \delta^{p})\\
&\lesssim (1+\|x\|_s^p) \cdot (\delta^{\frac{p}{2}} + \delta^{\frac{p}{2}}) +
(\delta^{p} + \delta^{p})\\
&\lesssim \delta^{\frac{p}{2}} \; (1+\|x\|^p_s).
\end{equs}
This finishes the proof of Corollary \ref{cor.alpha.bar}.
\end{itemize}
\end{proof}

\begin{proof}[Proof of Lemma \ref{lem.noise.estim}]
The martingale difference $\Gamma^{\delta}(x,\xi)$ defined in Equation
\eqref{eqn:approx:drift:noise}
can also be expressed as
\begin{equs}
\Gamma^{\delta}(x,\xi) = \xi + F(x,\xi)
\end{equs}
where the error term $F(x,\xi) = F_1(x,\xi) + F_2(x,\xi)$ is given by
\begin{equs}
F_1(x,\xi) &= (2 \tau \delta)^{-\frac{1}{2}} \big( (1 -2 \delta)^{\frac12} - 1
\big) \, \big( \gamma^{\delta}(x,\xi) 
\;-\; \EE_x[\gamma^{\delta}(x,\xi)]\big) x\\
F_2(x,\xi) &= \big( \gamma^{\delta}(x,\xi) - 1\big) \cdot \xi 
\;-\; \EE_x\big[ \gamma^{\delta}(x,\xi) \cdot \xi \big].
\end{equs}
We now prove that the quantity $F(x,\xi)$ satisfies
\begin{equs} \label{e.bound.FF}
\EE_x \Big[\|F(x,\xi)\|_s^2 \Big] \lesssim \delta^{\frac14} \; (1+\|x\|_s^2)
\end{equs}

\begin{itemize}
\item 
We have $\delta^{-\frac{1}{2}} \big( (1 - 2\delta)^{\frac12} - 1 \big) \lesssim
\delta^{\frac12}$ 
and $|\gamma^{\delta}(x,\xi)| \leq 1$. Consequently,
\begin{equs} \label{e.bound.FF.1}
\EE_x\Big[ \|F_1(x,\xi)\|_s^{2} \Big] \lesssim \delta \; \|x\|^{2}_s
\end{equs}

\item Let us now prove that $F_2$ satisfies 
\begin{equs} \label{e.bound.FF.2}
\EE_x\Big[ \|F_2(x,\xi)\|_s^{2} \Big] \lesssim\delta^{\frac14} \;
(1+\|x\|^{\frac12}).
\end{equs}
To this end, use the decomposition
\begin{equs}
\EE_x\Big[ \|F_2(x,\xi)\|_s^{2} \Big]
&\lesssim \EE_x \Big[ |\gamma^{\delta}(x,\xi) - 1|^{2} \cdot \|\xi\|_s^{2} \Big]
\;+\; \|\EE_x\big[ \gamma^{\delta}(x,\xi) \cdot \xi \big]\|_s^{2} \\
&= I_1  \;+\; I_2.
\end{equs}
The Cauchy-Schwarz inequality shows that
$I_1 \lesssim \EE_x \Big[ |\gamma^{\delta}(x,\xi) - 1|^{4} \Big]^{\frac12}$
where the Bernoulli 
random variable $\gamma^{\delta}(x,\xi)$ can be expressed as
$\gamma^{\delta}(x,\xi) = \indic_{\{ U < \alpha^{\delta}(x,\xi)\}}$ where $U
\dist \text{Uniform}(0,1)$ is
independent from any other source of randomness. 
Consequently 
\begin{equs}
\EE_x \Big[ |\gamma^{\delta}(x,\xi) - 1|^{4} \Big]
\;=\; \EE_x\big[ \indic_{\{\gamma^{\delta}(x,\xi) = 0\}} \; \big]
\;=\; 1 - \alpha^{\delta}(x)
\end{equs}
where the mean local acceptance probability $\alpha^{\delta}(x)$ is defined by 
$\alpha^{\delta}(x) = \EE_x[\alpha^{\delta}(x,\xi)] \in [0,1]$.
The convexity of the function $x \to |1-x|$ ensures that 
\begin{equs}
\big|1 - \alpha^{\delta}(x) \big|
\;=\; \big|1 - \EE_x\big[ \alpha^{\delta}(x,\xi)\big] \big|
\;\leq\; \EE_x\big[ \big|1 - \alpha^{\delta}(x,\xi) \big|\big] 
\;\lesssim\; \delta^{\frac12} \; (1+\|x\|)
\end{equs}
where the last inequality follows from Corollary \ref{cor.alpha.bar}.
This proves that $I_1 \lesssim \delta^{\frac14} \; (1+\|x\|^{\frac12})$.
To bound $I_2$, it suffices to notice
\begin{equs}
I_2 
&=\; \|\EE_x\big[ \gamma^{\delta}(x,\xi) \cdot \xi \big]\|_s^{2}
\;=\; \|\EE_x\big[ \big(\gamma^{\delta}(x,\xi)-1\big) \cdot \xi \big]\|_s^{2} \\
&\lesssim  \EE_x \Big[ |\gamma^{\delta}(x,\xi) - 1|^{2} \cdot \|\xi\|_s^{2}
\Big] 
\;=\; I_1
\end{equs}
so that $I_2 \lesssim I_1 \lesssim \delta^{\frac14} \; (1+\|x\|^{\frac12})$ and 
$\EE_x \Big[\|F_2(x,\xi)\|_s^{2} \Big] \lesssim \delta^{\frac14} \;
(1+\|x\|^{\frac12})$.
\end{itemize}
\noindent
Combining Equation \eqref{e.bound.FF.1} and \eqref{e.bound.FF.2} gives Equation
\eqref{e.bound.FF}.\\

\noindent
Let us now describe how Equations \eqref{e.noise.approx.3} and 
\eqref{e.noise.approx.4} follow from the estimate \eqref{e.bound.FF}.
\begin{itemize}
\item We have $\EE[\bra{\pphi_i, \xi}_s \bra{\pphi_j, \xi}_s ] = \bra{\pphi_i,
C_s \; \pphi_j}_s$ and
$\EE_x[\bra{\pphi_i, \Gamma^{\delta}(x,\xi)}_s \bra{\pphi_j,
\Gamma^{\delta}(x,\xi)}_s ] = \bra{\pphi_i, D^{\delta}(x) \; \pphi_j}_s$
with $\Gamma^{\delta}(x,\xi) = \xi + F(x,\xi)$. Consequently,
\begin{equs}
\bra{\pphi_i, D^{\delta}(x) \; \pphi_j}_s \;-\; \bra{\pphi_i,C_s \; \pphi_j}_s
&= \EE_x[\bra{\pphi_i, F(x,\xi)}_s \bra{\pphi_j, F(x,\xi)}_s] \\
&+\EE_x[\bra{\pphi_i, \xi}_s \bra{\pphi_j, F(x,\xi)}_s]\\
&+\EE_x[\bra{\pphi_i, F(x,\xi)}_s \bra{\pphi_j, \xi}_s]. 
\end{equs}
We have $|\bra{\pphi_i, F(x,\xi)}_s| \leq \|F(x,\xi)\|_s$ and the Cauchy-Schwarz 
inequality proves that
\begin{equs}
\EE_x[\bra{\pphi_i, F(x,\xi)}_s \bra{\pphi_j, \xi}_s]^2 
&\leq \EE_x[ \|F(x,\xi)\|_s \, \|\xi\|_s]^2\\
&\lesssim \EE_x[ \|F(x,\xi)\|^2_s ].
\end{equs}
It thus follows from Equation \eqref{e.bound.FF} that
\begin{equs}
| \bra{\pphi_i, D^{\delta}(x) \; \pphi_j}_s \;-\; \bra{\pphi_i,C_s \; \pphi_j}_s
|
&\lesssim \EE_x \big[ \|F(x,\xi)\|_s^2 \big] + \EE_x \big[ \|F(x,\xi)\|_s^2
\big]^{\frac12} \\
&\lesssim \delta^{\frac18} \, (1+\|x\|_s),
\end{equs}
finishing the proof of \eqref{e.noise.approx.1}.

\item We have $\tr_{\h^s}(C_s) = \EE[\|\xi\|_s^2]$ and
$\tr_{\h^s}(D^{\delta}(x)) = \EE[\|\Gamma^{\delta}(x,\xi)\|_s^2]$. Estimate
\eqref{e.bound.FF} thus shows that
\begin{equs}
|\tr_{\h^s}\big( D^{\delta}(x) \big) & \;-\; \tr_{\h^s}\big( C_s \big) |
= \big| \EE[ \|\Gamma^{\delta}(x,\xi)\|_s^2 - \|\xi\|_s^2] \big| \\
&\;=\; \big| \EE[ \|\xi \;+\; F(x,\xi) \|_s^2 - \|\xi\|_s^2] \big| \\
&\lesssim \big| \EE[ \bra{2 \xi \;+\; F(x,\xi),  F(x,\xi)}_s \big|
\;\lesssim\;  \EE[ \|2\xi \;+\; F(x,\xi)\|_s \; \|F(x,\xi)\|_s  ] \\
&\lesssim  \EE[ 4\|\xi\|_s^2  + \|F(x,\xi)\|^2_s ]^{\frac12} \cdot \EE[
\|F(x,\xi)\|_s^2 ]^{\frac12}\\
&\lesssim \Big( 1 + \delta^{\frac14} \; (1+\|x\|_s^2) \Big)^{\frac12} \cdot
\Big( \delta^{\frac18} \; (1+\|x\|_s)) \Big)
\;\lesssim\; \delta^{\frac18} (1+\|x\|_s^2),
\end{equs}
finishing the proof of \eqref{e.noise.approx.4}.
\end{itemize}
\end{proof}

\section{Proof of Lemma \ref{lem.accept.moves}} \label{app:accmoves}
Before proceeding to give the proof, let us give a brief proof sketch. The proof
of Lemma \ref{lem.accept.moves} consists in showing first that 
for any $\eps > 0$
one can find a ball of radius $R(\eps)$ around $0$ in $\h^s$,
\begin{equs}
B_0(R(\eps)) = \big\{x \in \h_s : \|x\|_s \leq R(\eps) \big\},
\end{equs}
such that with probability $1-2\eps$ we have $x^{k,\delta} \in B_0(R(\eps))$ 
and $y^{k,\delta} \in B_0(R(\eps))$ for all $0 \leq k \leq T\delta^{-1}$. 
As is described below, the existence of such a ball follows from the bound
\begin{equs} \label{e.max.bound}
\EE[\sup_{t \in [0,T]} \, \| x(t) \|_s \, ] \; < \; +\infty
\end{equs}
where $t \mapsto x(t)$ is the solution of the stochastic differential equation 
\eqref{eqn:limit:SPDE0}. For the sake of completeness, we include 
a proof of Equation \eqref{e.max.bound}.
The solution $t \mapsto x(t)$ of the stochastic differential equation
\eqref{eqn:limit:SPDE0} 
satisfies $x(t) = \int_0^t d\bigl(x(u)\bigr) \, du + \sqrt{2 \tau} \, W(t)$ 
for all $t \in [0,T]$ where the drift function $d(x) = -\big( x + C \nabla
\Psi(x) \big)$ 
is globally Lipschitz on $\h^s$, as described in Lemma
\ref{lem:lipshitz+taylor}. 
Consequently $\|d(x)\|_s \leq A(1+\|x\|_s)$ for some positive constant $A>0$. 
The triangle inequality then shows that
\begin{equs}
\|x(t)\|_s 
\leq A \int_0^t \big(1+\|x(u)\|_s \big) \, du + \sqrt{2 \tau} \| W(t) \|_s.
\end{equs}
By Gronwall's inequality we obtain
\begin{equs}
\sup_{[0,T]}\|x(t)\|_{s}\leq (A \, T + \sup_{[0,T]} \|W(t)\|_s) \, \big[ 1 + A
\, Te^{A \, T} \big].
\end{equs}
Since $\EE[ \sup_{[0,T]} \, \|W(t)\|_s] < \infty$, the bound 
\eqref{e.max.bound} is proved.

\begin{proof}[Proof of Lemma \ref{lem.accept.moves}]
The proof consists in showing that the the acceptance probability of the
algorithm 
is sufficiently close to $1$ so that approximation $t^{\delta}(k) \approx k$
holds.
The argument can be divided into $3$ main steps. 
In the first part, 
we show that we can find a finite ball $B(0,R(\eps))$ such that the trajectory
of the Markov chain 
$\{x^{k,\delta}\}_{k \leq T \delta^{-1}}$ remains in this ball with probability
at least $1-2\eps$. 
This observation is useful since the function $\Psi$ is Lipschitz on any ball of
finite radius in $\h^s$.
In the second part, using the fact that $\Psi$ is Lipschitz on $B(0,R(\eps))$,
we find 
a lower bound for the acceptance probability $\alpha^{\delta}$. 
Then, in the last step, we use a moment estimate to prove that one can make 
the lower bound uniform on the interval $0 \leq k \leq T \delta^{-1}$.

\begin{itemize}

\item {\bf Restriction to a Ball of Finite Radius}\\
First, we show that with high probability the trajectory 
of the MCMC algorithm stays in a ball of finite radius. The functional 
$x \mapsto \sup_{t \in [0,T]} \|x(t)\|_s$ is continuous on $C([0,T], \h_s)$ and 
$\EE\big[ \sup_{t \in [0,T]} \|x(t)\|_s \big] < \infty$ for $t \mapsto x(t)$
following the 
stochastic differential equation \eqref{eqn:limit:SPDE0}, as proved 
in Equation \eqref{e.max.bound}. 
Consequently, the weak convergence 
of $z^{\delta}$ to the solution 
of \eqref{eqn:limit:SPDE0} encapsulated in Theorem
\ref{thm:main} shows that 
$\EE\big[ \sup_{k < T \delta^{-1}} \|x^{k,\delta}\|_s \big]$ can be bounded by
a 
finite universal constant independent from $\delta$.
Given $\eps > 0$, Markov inequality thus shows that one can find a radius
$R_1=R_1(\eps)$ 
large enough so that the inequality
\begin{equs} \label{e.bound.x_k}
\PP \big[ \| x^{k,\delta} \|_s < R_1 \quad \text{for all} \quad 0 \leq k \leq T
\delta^{-1} \big] > 1-\eps 
\end{equs}
for any $\delta \in (0,\frac12)$.
By Fernique's Theorem there exists $\alpha>0$ such that
$\EE[e^{\alpha \|\xi\|_s^2}] < \infty$. This implies that 
$\PP[\| \xi \|_s > r] \lesssim e^{-\alpha r^2}$. Therefore, if $\{\xi_k\}_{k
\geq 0}$ are 
$\iid$ Gaussian random variables distributed as $\xi \dist \pi_0$,
the union bound shows that
\begin{equs}
\PP \big[ \| \sqrt{\delta} \xi_k \|_s \leq r  \quad \text{for all} \quad 0 \leq
k \leq T \delta^{-1}  \big] 
\gtrsim 1 -  T \delta^{-1} \exp(-\alpha \delta^{-1} r^2).
\end{equs}
This proves that one can choose $R_2 = R_2(\eps)$ large enough 
in such a manner that
\begin{equs} \label{e.bound.xi}
\PP \big[ \| \sqrt{\delta} \xi_k \|_s < R_2 \quad 
\text{for all} \quad 0 \leq k \leq T \delta^{-1} \big] > 1-\eps 
\end{equs}
for any $\delta \in (0,\frac12)$.
At temperature $\tau > 0$ the MCMC proposals are given by 
$y^{k,\delta} = (1-2 \delta)^{\frac12} x^{k,\delta} + (2\delta \tau)^{\frac12}
\, \xi_k$. 
It thus follows from
the bounds \eqref{e.bound.x_k} and \eqref{e.bound.xi}  that 
with probability at least $(1-2\eps)$ the vectors $x^{k,\delta}$ and
$y^{k,\delta}$ 
belong to the ball $B_0(R(\eps)) = \{ x \in \h_s : \|x\|_s < R(\eps)\}$ 
for $0 \leq k \leq T \delta^{-1}$
where radius $R(\eps)$ is given by $R(\eps) = R_1(\eps)+R_2(\eps)$.

\item {\bf Lower Bound for Acceptance Probability}\\
We now give a lower bound for the acceptance probability
$\alpha^{\delta}(x^{k,\delta}, \xi^k)$ 
that the move $x^{k,\delta} \to y^{k,\delta}$ is accepted.
Assumptions \ref{ass:1} state that $\| \nabla \Psi(x) \|_{-s} \lesssim 1 +
\|x\|_s$.
Therefore, the function $\Psi: \h^s \to \RR$ is Lipschitz on $B_0(R(\eps))$,
\begin{equs}
\|\Psi\|_{\lip,\eps} \; \eqdef \; \sup \Big\{ \frac{|\Psi(y) -
\Psi(x)|}{\|y-x\|_s} 
\,:\, x,y \in B_0(R(\eps))\Big\} < \infty.
\end{equs}
One can thus bound the acceptance probability 
$\alpha^{\delta}(x^{k,\delta}, \xi^{k}) = 1 \wedge \exp\big(
-\tau^{-1}[\Psi(y^{k,\delta}) - \Psi(y^{k,\delta})]\big)$ for $x^{k,\delta},
y^{k,\delta} \in B_0(R(\eps))$. Since the function 
$z \mapsto 1 \wedge e^{-\tau^{-1}z}$ is Lipschitz with constant $\tau^{-1}$, the
definition 
of $\|\Psi\|_{\lip,\eps}$ shows that the bound 
\begin{equs}
1-\alpha^{\delta}(x^{k,\delta}, \xi^{k}) 
&\leq \tau^{-1} \, \| \Psi \|_{\lip,\eps} \, \|y^{k,\delta} - x^{k,\delta}\|_s
\\
&\leq \tau^{-1} \, \| \Psi \|_{\lip,\eps} \, \Big\{ [(1-2\delta)^{\frac12} -
1]\, \|x^{k,\delta}\|_s 
+ (2\delta \tau)^{\frac12} \, \|\xi^k\|\Big\}\\
&\lesssim \sqrt{\delta} \, (1 +  \|\xi^k\|_s) 
\end{equs}
holds for every $x^{k,\delta}, y^{k,\delta} \in B_0(R(\eps))$. Hence, 
there exists a constant $K=K(\eps)$ such that 
$\what{\alpha}^{\delta}(\xi^k) = 1 - K  \sqrt{\delta} \, (1 +  \|\xi^k\|_s)$ 
satisfies $\alpha^{\delta}(x^{k,\delta}, \xi^{k}) >
\what{\alpha}^{\delta}(\xi^k)$ 
for every $x^{k,\delta}, y^{k,\delta} \in B_0(R(\eps))$.
Since the trajectory 
of the MCMC algorithm stays in the ball $B_0(R(\eps))$ with probability at
least 
$1-2\eps$ the inequality
\begin{equs}
\PP[\alpha^{\delta}(x^{k,\delta}, \xi^{k}) > \what{\alpha}^{\delta}(\xi^k)
 \quad \text{for all} \quad 0 \leq k \leq T \delta^{-1} ] \;>\; 1 - 2 \eps.
 \end{equs}
holds for every $\delta \in (0,\frac12)$.

\item {\bf Second Moment Method}\\
To prove that $t^{\delta}(k)$ does not deviate too much from $k$, we show 
that its expectation satisfies $\EE[t^{\delta}(k)] \approx k$ and we then 
control the error by bounding the variance. Since the Bernoulli random variable 
$\gamma^{k,\delta} = \textrm{Bernoulli}(\alpha^{\delta}(x^{k,\delta}\xi^k))$ are
not 
independent, the variance of $t^{\delta}(k) = \sum_{l \leq k} \gamma^{l,\delta}$
is 
not easily computable. We thus introduce $\iid$ auxiliary random variables 
$\what{\gamma}^{k,\delta}$ such that 
\begin{equs}
 \sum_{l \leq k} \what{\gamma}^{l,\delta} =  \what{t}^{\delta}(k) 
 \quad \approx \quad t^{\delta}(k) = \sum_{l \leq k} \gamma^{l,\delta}.
\end{equs}
As described below, the behaviour of $\what{t}^{\delta}(k)$ is readily
controlled since 
it is a sum of $\iid$ random variables. The proof then exploits the fact that
 $\what{t}^{\delta}(k)$ is a good approximation of $t^{\delta}(k)$.\\
 
 \noindent
The Bernoulli random variables $\gamma^{k,\delta}$ can be described as 
$\gamma^{k,\delta} = \indic\big( U_k < \alpha^{\delta}(x^{k,\delta}\xi^k)\big)$
where 
$\{U_k\}_{k \geq 0}$ are $\iid$ random variables uniformly distributed on
$(0,1)$. 
As a consequence, with probability at least $1-2 \eps$, the random variables 
$\what{\gamma}^{k,\delta} = \indic\big( U_k < \what{\alpha}^{\delta})$ satisfy
$\gamma^{k,\delta} \geq \what{\gamma}^{k,\delta}$
for all $0 \leq k \leq T \delta^{-1}$. 
Therefore, with probability at least $1-2\eps$, we have 
$t^{\delta(k)} \geq \what{t}^{\delta(k)}$ for all $0 \leq k \leq T \delta^{-1}$
where
$\what{t}^{\delta(k)} = \sum_{l \leq k} \what{\gamma}^{l, \delta}$. 
Consequently, since $t^{\delta(k)} \leq k$, to prove Lemma
\ref{lem.accept.moves} it suffices to show instead that the following limit in
probability holds, 
\begin{equs} \label{e.lim.prob.modified}
\lim_{\delta \to 0} \quad \sup \big\{ \delta \cdot \big| \what{t}^{\delta}(k) -k
\big|  \; : 0 \leq k \leq T \delta^{-1} \big\} = 0.
\end{equs}
Contrary to the random variables $\{ \gamma^{k, \delta} \}_{k \geq 0}$, the
random variables 
$\{ \what{\gamma}^{k, \delta} \}_{k \geq 0}$ are $\iid$ and are thus easily
controlled. 
By Doob's inequality we have
\begin{equs}
\PP \Big[ \sup \big\{ \delta \cdot \big| \what{t}^{\delta}(k) -
\EE[\what{t}^{\delta}(k)] \big|  \,:\, 
0 \leq k \leq T \delta^{-1} \big\} \,>\,\eta \Big] 
\;\leq\; 2 \, \frac{\textrm{Var}\big( \what{t}^{\delta}(T\delta^{-1})
\big)}{(\delta^{-1} \eta)^2} 
\;\leq\; 2 \, \frac{\delta T}{\eta^2}.
\end{equs}
Since $\EE[\what{t}^{\delta}(k)] = k \cdot \big\{ 1 - K  \sqrt{\delta} \, (1 + 
\EE[\|\xi^k\|_s]) \big\}$,  
Equation \eqref{e.lim.prob.modified} follows. This finishes the proof of Lemma 
\ref{lem.accept.moves}.
\end{itemize}
\end{proof}

%
%
\begin{acknowledgements}
The authors thank an anonymous referee for constructive comments. We are grateful to David Dunson for
the his comments on the implications of theory, Frank Pinski 
for helpful discussions concerning the behaviour of the
quadratic variation; these discussions crystallized the need
to prove Theorem \ref{th.fluid.limit}. NSP gratefully acknowledges
the NSF grant DMS 1107070. AMS is grateful 
to EPSRC and ERC for financial support. Parts of this work
was done when AHT was visiting the department of Statistics at 
Harvard university. The authors thank the department of statistics, Harvard
University
for its hospitality.
\end{acknowledgements}

\bibliographystyle{spmpsci}      
\bibliography{mcmcgrad}   

%
%

\end{document}